\documentclass[12pt,oneside,english]{amsart}

\usepackage{style_andris}

\title{Leading terms of generalized Pl\"ucker formulas}
\author{Andr\'as P. Juh\'asz}
\address{Alfr\'ed R\'enyi Institute of Mathematics Budapest, Hungary}
\email{juhasz.andris@gmail.com}


\begin{document}
\begin{abstract}
Generalized Pl\"ucker numbers are defined to count certain types of tangent lines of generic degree $d$
complex projective hypersurfaces. They can be computed by identifying them as coefficients of
$\GL(2)$-equivariant cohomology classes of certain invariant subspaces of $\Pol^d(\mathbb{C}^2)$,
the so-called coincident root strata.
In an earlier paper L\'aszl\'o M. Feh\'er and the author gave a new, recursive method for calculating 
these classes. Using this method, we showed that---similarly to the classical Pl\"ucker formulas counting
the bitangents and
flex lines of a degree $d$ plane curve---generalized Pl\"ucker numbers are polynomials in the degree $d$.

In this paper, by further analyzing our recursive formula, we determine the leading terms of all the
generalized Pl\"ucker formulas.
\end{abstract}

\maketitle
\tableofcontents

\section{Introduction}\label{sec_introduction}
This paper can be viewed as a companion paper of
\cite{feher_juhasz2023plucker} by L\'aszl\'o M. Feh\'er and the author.
It contains the asymptotic analysis of the generalized Pl\"ucker formulas that were introduced there:
We present a proof for a statement---which was already announced 
(\cite[Thm.~4.8]{feher_juhasz2023plucker})---that describes their degrees, and we calculate their 
leading coefficients.

\medskip

I am grateful to L\'aszl\'o M. Feh\'er for the joint work that resulted in
\cite{feher_juhasz2023plucker}, for his encouragement to look into the leading coefficients
and for his advice on editing this paper.
\medskip

In the 1830s, Pl\"{u}cker showed that a smooth degree $d$ complex projective plane curve has
\begin{equation}\label{eq_classicalPlucker22_3}
\Pl_{2,2}(d)=\frac12d(d-2)(d-3)(d+3) \quad \text{ and } \quad \Pl_{3}(d)=3d(d-2),
\end{equation}
bitangents and flex lines respectively. His classical formulas also cover the cases of singular curves,
but we only study the generic case.

A generic degree $d$ plane curve has no tritangents, i.e. tangent lines with three distinct points of tangency.
For tritangents to appear we have to increase the dimension and consider complex projective
hypersurfaces of dimension at least two.

Generalizing bitangent, flex and tritangent lines, we define a tangent line to be of type $\lambda$
if at the points of tangency the line intersects the hyperplane with multiciplities given by elements of 
the partition $\lambda$.
For example, bitangents correspond to $\lambda=(2,2)$.
For each $\lambda$, a generic hypersurface of big enough dimension has tangent lines of type $\lambda$,
most of the time infinitely many.
By adding certain linear conditions, we can obtain finite subsets of type $\lambda$ tangent lines of
generic degree $d$ hypersurfaces.
We defined generalized Pl\"{u}cker numbers to be the cardinality of these finite subsets.

For example, corresponding to the partition $\lambda=(2)$, a generic degree $d$ plane curve has infinitely many
ordinary tangents, among which there are
\begin{equation}\label{eq_classicalPlucker2}
\Pl_{2;1}(d)=d(d-1) 
\end{equation}
that passes through a generic point of $\P(\C^3)$. 
In other words, the degree of the dual curve is $d(d-1)$.

\medskip

A key result of \cite{feher_juhasz2023plucker} is a recursive formula for the generalized Pl\"{u}cker numbers.
Using this formula, we prove that the $d$-dependence of all the generalized Pl\"{u}cker numbers is polynomial.
To give a closed formula that describes all these polynomials doesn't seem feasible at this point.
In this paper we restrict ourselves to the analysis of their leading term.
In \cite{feher_juhasz2023plucker} we have already shown that 
for each $\lambda$
the $d$-degrees of all the Pl\"{u}cker formulas corresponding to $\lambda$
are at most $|\lambda|=\sum \lambda_i$ (\cite[Thm.~4.6]{feher_juhasz2023plucker}).
We also calculated the leading coefficients of those whose $d$-degrees reach this upper bound
$|\lambda|$ (\cite[Thm.~6.1]{feher_juhasz2023plucker}).
In this paper we carry on with our investigation and determine the leading terms
of all the Pl\"{u}cker formulas (Theorem \ref{thrm_leadingtermPl}).

\medskip

In the next Section we first give the precise definitions of the notions outlined above.
Then, we will finally be able to state our main result, Theorem \ref{thrm_leadingtermPl}.

\subsection{The definition and polynomial \texorpdfstring{$d$}{d}-dependence of generalized Pl\"{u}cker numbers}
Let $f\in \Pol^d(\C^n)$ be a nonzero homogeneous polynomial of degree $d$ in $n$ variables.
It defines a hypersurface $Z_f=(f=0)$ in $\P(\C^n)$. Let
\[ \lambda=(\lambda_1\geq\lambda_2\geq\cdots\geq\lambda_k)=(2^{e_2},\dots,r^{e_r})\]
be a partition without $1$'s and $d\geq|\lambda|$. A line in $\P(\C^n)$ is called a tangent line of type
$\lambda$ to $Z_f$ if it has $e_2$ ordinary tangent points, $e_3$ flex points, etc.
A formal definition can be given the following way.
Projective lines $[V]$ in $\P(\C^n)$ correspond to affine planes $V^2$ of $\Gr_2(\C^n)$.
\begin{definition} The projective line $[V]$ is called a \emph{tangent line of type $\lambda$ to $Z_f$} (or
  \emph{$\lambda$-line} for short) if
  \[ f|_V=\prod_{i=1}^{k} \left( f_i^{\lambda_i} \right) \prod_{j=|\lambda|+1}^d \left( f_j\right), \]
where $f_i,f_j:V\to \C$ are linear and no two of them are scalar multiples of each other.
\end{definition}
For a given polynomial $f\in \Pol^d(\C^n)$ let us denote by
\[ \mathcal{T}_\lambda Z_f:=\{ \text{tangent lines of type $\lambda$ to $Z_f$}\} \subset \gr_2(\C^n),\]
the \emph{variety of tangent lines of type $\lambda$ to $Z_f$}.
Strictly speaking, $\mathcal{T}_\lambda Z_f$ is locally closed set; we will use the term variety in this
broader sense.

Note here that although $\lambda$-lines are well-defined for the partition $\lambda=\emptyset$,
those are not, in the usual sense, tangent to $Z_f$.
Hopefully, this will not cause any confusion. Also, we will not examine
$\mathcal{T}_\lambda Z_f$ for $\lambda=\emptyset$.

A simple dimension counting gives that for a generic polynomial $f \in \Pol^{d}(\mathbb{C}^n)$
the codimension of $\mathcal{T}_\lambda Z_f$ in $\Gr_2(\C^n)$ is
$\sum_{i=1}^{k}(\lambda_i-1)=\sum_{j=2}^{r}(j-1)e_j$, see \cite{fnr-root} and \cite{feher_juhasz2023plucker}
for more details. For this reason, we introduce the partition
\[\tilde{\lambda}:=(\lambda_1-1,\lambda_2-1,\dots,\lambda_k-1),\]
the \emph{reduction of $\lambda$}. Then
\[\codim\big( \mathcal{T}_\lambda Z_f \subset \Gr_2(\C^n) \big)=|\tilde{\lambda}|.\]

The premise generic is crucial for our approach to work and, hence, it appears in all our statements
about varieties of $\lambda$-lines.
In algebraic geometry, a claim is defined to hold for a \emph{generic} element,
if the elements satisfying the claim 
form a nonempty open subset.
For example, a generic hypersurface is smooth.

Note, in particular, that the codimension of
$\mathcal{T}_\lambda Z_f \subset \Gr_2 (\C^n)$ for $f \in \Pol^{d}(\mathbb{C}^n)$ generic
is independent of $d$ (and $n$).
We are mostly interested in the $d$-dependence. But see Remark \ref{rmrk_ndependencePl}.

\medskip

The dimension of the Grassmannian $\Gr_2(\C^n)$ is $2(n-2)$.
If for a partition $\lambda$ the corresponding codimension $|\tilde{\lambda}|$ matches this dimension,
then for a generic $f \in \Pol^{d}(\mathbb{C}^n)$ the variety $\mathcal{T}_\lambda Z_f$
is zero-dimensional, and we can ask its cardinality:

\begin{definition} Let $\lambda$ be a nonempty partition without $1$'s such that $2(n_0-2)=|\tilde{\lambda}|$ for some $n_0$. Then the \emph{Pl\"ucker number} $\Pl_\lambda(d)$ for $d\geq|\lambda|$ is defined as the number of type $\lambda$ tangent lines to a generic degree $d$ hypersurface in $\P(\C^{n_0})$.
\end{definition}

This explains why in \eqref{eq_classicalPlucker22_3} we used $\Pl_{2,2}(d)$ and $\Pl_{3}(d)$ for the classical
Pl\"{u}cker formulas.
Note that, for typographical reasons, we omit brackets from the indices.

If the dimension of $\mathcal{T}_\lambda Z_f \subset \Gr_2(\C^n)$ is positive,
we add linear conditions to obtain enumerative problems about $\lambda$-lines.
This motivates the following.
\begin{definition}
  Let $\lambda$ be a nonempty partition without 1's.
  Choose $n_0$ and $0 \leq i\leq |\tilde\lambda|$ such that $|\tilde\lambda|+i=2(n_0-2)$.
  We define the \emph{Pl\"ucker number} $\Pl_{\lambda;i}(d)$ for $d\geq|\lambda|$ as the number of
  $\lambda$-lines of a generic degree $d$ hypersurface in $\P(\C^{n_0})$ intersecting a generic
  $(i+1)$-codimensional projective subspace.
\end{definition}

For $\Pl_{\lambda;0}(d)$ we recover the previous definition: $\Pl_{\lambda;0}(d)=\Pl_{\lambda}(d)$.
This definition is consistent with the notation $\Pl_{2;1}(d)$ in \eqref{eq_classicalPlucker2}
showing the number of ordinary tangent lines of generic degree $d$ plane curves
passing through a generic point.

\begin{example} \label{pl22;2}
For bitangent lines we also have
\[ \Pl_{2,2;2}(d)=\frac{1}{2} d \left( d-1 \right)  \left( d-2 \right)  \left( d-3 \right),\]
 the number of bitangent lines of a generic degree $d$ surface in $\P(\C^4)$ going through a point.
\end{example}

Let us remark that the Pl\"ucker numbers $\big\{ \Pl_{\lambda;i}(d) : 0 \le i \le |\tilde{\lambda}|,
  i \equiv |\tilde{\lambda}| \, (\!\!\!\! \mod 2) \big\}$ are actually defined to encode the cohomology class
  of
  $\overline{\mathcal{T}_\lambda Z_f}$:
Given a partition $\lambda$ without 1's, $n \ge |\tilde{\lambda}|+2$
and $f \in \Pol^{d}(\mathbb{C}^n)$ generic, then
\begin{equation}\label{eq_classofTlambdaarePluckernumbers}
\left[\,\overline{\mathcal{T}_\lambda Z_f}\subset \gr_2(\C^n) \right] =\sum_{j=0}^{\lfloor|\tilde\lambda|/2\rfloor}
\Pl_{\lambda;|\tilde\lambda|-2j}(d)    s_{|\tilde\lambda|-j,j},
\end{equation}
where the $s_{|\tilde\lambda|-j,j}$'s denote Schur polynomials, i.e. cohomology classes of the Schubert 
varieties, see \cite{feher_juhasz2023plucker} for more details.

\begin{remark}\label{rmrk_ndependencePl}
  That the right-hand side of \eqref{eq_classofTlambdaarePluckernumbers} doesn't contain the parameter $n \ge |\tilde{\lambda}|+2$ reflects the fact that
the Pl\"ucker number $\Pl_{\lambda;i}(d)$ solves a family of enumerative problems (even though it is  
defined as a solution to one for a specific $n_0$ with $2(n_0-2)=|\tilde{\lambda}|+i$):
Elementary geometric considerations imply that if $n \ge n_0$, then
  $\Pl_{\lambda;i}(d)$ is the number of $\lambda$-lines of a generic degree $d$ hypersurface
  in $\mathbb{P}(\mathbb{C}^n)$ intersecting a generic $(n-n_0+i+1)$-codimensional projective subspace $A$
  and contained in a generic $(n_0-1)$-dimensional projective subspace $B$ such that $A \subset B$.
\end{remark}

\medskip

In \cite{feher_juhasz2023plucker} we prove that
the $d$-dependence of generalized Pl\"ucker numbers is
polynomial. This was our motivation for having $d$ as a variable in our notation:
\begin{theorem}[{\cite[Thm.~2.4.3]{feher_juhasz2023plucker}}]
  \label{thrm_pluckernumberspolynomial}
  The Pl\"ucker numbers $\Pl_{\lambda;i}(d)$ for $0\leq i\leq |\tilde\lambda|$ and $i\equiv |\tilde\lambda|$
  ${(\!\!\!\! \mod 2)}$ are polynomials in $d$: there is a unique polynomial $p(d) \in \mathbb{Q}[d]$ such that $\Pl_{\lambda;i}(d)=p(d)$ for $d\geq|\lambda|$.
\end{theorem}
We will refer to these polynomials as \emph{Pl\"{u}cker formulas} and denote them the same way,
$\Pl_{\lambda;i}(d) \in \mathbb{Q}[d]$ as we denoted the Pl\"{u}cker numbers
(values of Pl\"{u}cker formulas at specific $d$'s).

Now we can state the main theorem of the paper.
\begin{theorem}\label{thrm_leadingtermPl}
  Let $\lambda_1$ be the largest number in the partition $\lambda=(2^{e_2},\dots,r^{e_r})$. Then
  \begin{multline*}
    \text{the leading term of } \Pl_{\lambda;|\tilde{\lambda}|-2j}(d)=\\
    \frac{1}{\prod_{i=2}^r \left( e_i!\right)}
    \begin{cases} 
      K_{(|\tilde{\lambda}|-j,j),\tilde{\lambda}}\, d^{|\lambda|} &
      \text{if } j\leq |\tilde{\lambda}|-\lambda_1+1,
\\[1em]
      \stir{\lambda_1}{\lambda_1-\left(j-\left( |\tilde{\lambda}|-\lambda_1+1 \right) \right)} d^{|\lambda|-\left(j-\left( |\tilde{\lambda}|-\lambda_1+1 \right) \right)} &
      \text{if } j> |\tilde{\lambda}|-\lambda_1+1,
    \end{cases}
  \end{multline*}
  where the $K_{\mu,\nu}$'s denote Kostka numbers and the $\stir{m}{m-k}$'s are Stirling numbers of the 
  first kind.
\end{theorem}
In other words, for each partition $\lambda$ there exists a threshold
$\theta(\lambda)=\min(\lfloor|\tilde\lambda|/2\rfloor,|\tilde{\lambda}|-\lambda_1+1)$
  such that $\Pl_{\lambda;|\tilde{\lambda}|-2j}(d)$ has
degree $|\lambda|$ for $j=0,\dots,\theta(\lambda)$, then by increasing $j$ by one, the degree
drops by one.
Also, for $j=0,\dots,\theta(\lambda)$ the leading coefficients of $\Pl_{\lambda;|\tilde{\lambda}|-2j}(d)$
can be described using Kostka numbers. For partitions $\mu$ and $\nu$ $K_{\mu,\nu}$ counts the
semistandard Young tableaux of shape $\mu$ and weight $\nu$. For the latter $j$'s these leading coefficients
are described in terms of Stirling numbers of the first kind that can be defined via elementary
symmetric polynomials $\sigma_k$ as
\[ \stir{m}{m-k}=\sigma_k(1,2,\dots,m-1).  \]
 
The combinatorial nature of these leading coefficients supports the claim that it seems to be difficult to
give closed formulas for the Pl\"{u}cker formulas $\Pl_{\lambda;|\tilde{\lambda}|-2j}(d)$ in terms of $\lambda$
and $d$.

\begin{example}
For $\lambda=(10,2,2)$ we have $|\lambda|=14$, $|\tilde{\lambda}|=11$, $\lambda_1=10$
and $\theta(\lambda)=|\tilde{\lambda}|-\lambda_1+1=2$, implying that
\[ \deg\big(\Pl_{10,2,2;11}(d)\big)=\deg\big(\Pl_{10,2,2;9}(d)\big)=\deg\big(\Pl_{10,2,2;7}(d)\big)=14, \]
and
\[ \deg\big(\Pl_{10,2,2;5}(d)\big)=13,\ \ \deg\big(\Pl_{10,2,2;3}(d)\big)=12,\ \ \deg\big(\Pl_{10,2,2;1}(d)\big)=11.\]
\end{example}

If $\lambda_1$ is not much bigger than the other $\lambda_i$, exactly
if $\lambda_1\leq \lceil|\tilde{\lambda}|/2\rceil +1$, then all the Pl\"ucker formulas
$\Pl_{\lambda;|\tilde{\lambda}|-2j}(d)$ have degree
$|\lambda|$. We saw this in \eqref{eq_classicalPlucker22_3} and Example \ref{pl22;2} for the bitangents:
both $\Pl_{2,2;0}$ and $\Pl_{2,2;2}$ have degree $|\lambda|=4$.
A slightly bigger example is $\lambda=(4,3,2)$, where all the Pl\"ucker formulas have degree $|\lambda|=9$.

\medskip

Section \ref{sec_proofdegreesofPlnumbers} will be dedicated to the proof of Theorem \ref{thrm_leadingtermPl}.
The proof is based on a recursive formula for the Pl\"ucker numbers.
That recursion, however, is deduced using our preferred language: in terms of equivariant cohomology
classes of coincident root strata,
$\big[ \, \overline Y_{\lambda}(d) \subset \Pol^{d}(\mathbb{C}^2)\big]_{\GL(2)}$.
These classes are universal, which implies that
for each $n$ and $f \in \Pol^{d}(\mathbb{C}^n) $ generic the class
$\left[\,\overline{\mathcal{T}_\lambda Z_f}\subset \gr_2(\C^n) \right]$ can be deduced from 
them.
In fact, for $n \ge |\tilde{\lambda}|+2$ the classes $\left[\,\overline{\mathcal{T}_\lambda Z_f}\subset
\Gr_2(\mathbb{C}^n)\,\right]$ and
$\big[ \, \overline Y_{\lambda}(d) \subset \Pol^{d}(\mathbb{C}^2)\big]_{\GL(2)}$
contain the exact same information.

In the first part of Section \ref{sec_recapequivclassofCRS} we recall the definition of the coincident root
stratum $Y_{\lambda}(d) \subset \Pol^{d}(\mathbb{C}^2)$, and we explain how one can deduce 
$\left[\,\overline{\mathcal{T}_\lambda Z_f}\subset \Gr_2(\mathbb{C}^n)\,\right]$
from its equivariant cohomology class.
Those not interested can take 
\[
\big[ \, \overline Y_{\lambda}(d) \subset \Pol^{d}(\mathbb{C}^2)\big]_{\GL(2)}=\sum_{j=0}^{\lfloor|\tilde\lambda|/2\rfloor}
\Pl_{\lambda;|\tilde\lambda|-2j}(d)    s_{|\tilde\lambda|-j,j},
\]
---where the $s_{|\tilde\lambda|-j,j}$'s are Schur polynomials in some variables $a$ and $b$\,---as the definition, and
jump to the second part starting from Theorem \ref{recursion4Y}, where we state 
the recursive formula and its corollaries necessary to prove Theorem \ref{thrm_leadingtermPl}.

\subsection{A recursion for generalized Pl\"{u}cker numbers in terms of equivariant cohomology classes of coincident root strata}\label{sec_recapequivclassofCRS}
The vector space
\[\Pol^d(\C^2):=\{\mbox{homogeneous polynomials of degree $d$ in two variables}\}\]
admits a stratification into the so-called coincident root strata:

\begin{definition} Let $\lambda=(\lambda_1\geq\lambda_2\geq\cdots\geq\lambda_k)$ be a partition without
  $1$'s and $d\geq |\lambda|$. Then the \emph{coincident root stratum of $\lambda$} is
  \[   Y_\lambda(d)  :=  \left\{ g \in \Pol^d(\C^2) :  g=\prod_{i=1}^{k} \left( g_i^{\lambda_i} \right)
  \prod_{j=|\lambda|+1}^d \left(  g_j \right) \right\},    \]
where $g_i,g_j:\C^2 \to \C$ are nonzero, linear and no two of them are scalar multiples of each other.
\end{definition}
The strata $Y_\lambda(d)$ together with $\{0\}$ gives a stratification of $\Pol^d(\C^2)$.
For example,
\[ \Pol^{4}(\mathbb{C}^2)=Y_{\emptyset}(4) \amalg Y_{2}(4) \amalg Y_{2,2}(4) \amalg Y_{3}(4) \amalg
Y_{4}(4) \amalg \{0\}. \]
For each $\lambda$ the corresponding stratum $Y_\lambda(d)$ has codimension $|\tilde{\lambda}|$
(\cite{fnr-root}), and 
it is invariant for the $\GL(2)$-action on
$\Pol^d(\C^2)\cong \Sym^d\left({\mathbb{C}^2}^\vee \right)$ coming from the standard
representation of $\GL(2)$ on $\mathbb{C}^2$.
The latter implies, see e.g. \cite{totaro}, that (the closure of) every stratum admits a $\GL(2)$-equivariant
cohomology class
\[ \left[ \, \overline Y_{\lambda}(d) \subset \Pol^{d}(\mathbb{C}^2)\right]_{\GL(2)} \in 
H^{*}_{\GL(2)}\left( \Pol^{d}(\mathbb{C}^2) \right) \cong \mathbb{Z}[c_1,c_2], \]
where the $c_i$'s denote some Chern classes, see \cite{feher_juhasz2023plucker} for more details.
Sometimes we drop the group $\GL(2)$ and the ambient space $\Pol^{d}(\mathbb{C}^2)$ from our notation,
and simply write
$\left[ \, \overline Y_{\lambda}(d) \right]$ for 
$\left[ \, \overline Y_{\lambda}(d) \subset \Pol^{d}(\mathbb{C}^2)\right]_{\GL(2)}$.

\medskip

``Using Kleiman's theory of multiple point formulas (\cite{kleiman1977enumtheorysingularities,
kleiman1981multiplepoint_iteration,Kleiman1982multiplepoint_formaps}) Le Barz in \cite{lebarz-formules} and Colley  in \cite{colley1986contact} calculated examples of Pl\"ucker numbers.

Kirwan gave formulas for the $\SL(2)$-equivariant cohomology classes of coincident root strata in \cite{kirwan}. The first formula for the $\GL(2)$-equivariant cohomology classes $\left[\,\overline Y_\lambda(d)\right]$ was given in \cite{fnr-root}. Notice that the $\SL(2)$-equivariant cohomology classes are obtained from the $\GL(2)$-equivariant ones by substituting zero into $c_1$, therefore they do not determine the corresponding Pl\"ucker numbers.
Soon after, a different formula was calculated with different methods in \cite{balazs-tezis}. These formulas don't seem to be useful for proving polynomiality in $d$. In 2006 in his unpublished paper \cite{kazarian}  Kazarian deduced a formula in a form of a generating function from his theory of multisingularities of Morin maps based on Kleiman's theory of multiple point formulas. This formula shows the polynomial dependence but further properties do not seem to follow easily. He also calculated several Pl\"ucker formulas $\Pl_\lambda(d)$. The paper \cite{spink-tseng} of Spink and Tseng also develops a method to calculate the $\GL(2)$-equivariant cohomology classes $\left[\,\overline Y_\lambda(d)\right]$. One of their main goals is to establish relations between these classes.``(\cite{feher_juhasz2023plucker})

\medskip

Such equivariant cohomology classes are universal polynomials:
cohomology classes of (closures of) generic $Y_\lambda(d)$-loci can be deduced from them.
For $f \in \Pol^{d}(\mathbb{C}^n)$ the variety $\mathcal{T}_\lambda Z_f \subset \Gr_2(\C^n)$ is the
$Y_\lambda(d)$-loci of the section $\sigma_f: W \mapsto f|_W$ of the vector bundle
$\Pol^{d}(S) \to \Gr_2(\C^n)$, where $S \to \Gr_2(\C^n)$ denotes the tautological bundle.
For a generic polynomial $f \in \Pol^{d}(\mathbb{C}^n)$ the section $\sigma_f$ is
transversal to the subbundle of $\Pol^{d}(S)$ consisting of $Y_\lambda(d)$-points.
This gives that
\begin{proposition}[{\cite[Cor.~2.3]{feher_juhasz2023plucker}}]
\label{cor:equi2nonequi}
For a generic polynomial $f \in \Pol^{d}(\mathbb{C}^n)$
the cohomology class $\left[\,\overline{\mathcal{T}_\lambda Z_f}\subset \Gr_2(\mathbb{C}^n)\,\right]$
is obtained from the equivariant class $\left[\,\overline Y_\lambda(d) \subset \Pol^{d}(\mathbb{C}^2)
\right]_{\GL(2)}\in \Z[c_1,c_2]$
by substituting $c_i(S^\vee)$ into $c_i$ for $i=1,2$.
\end{proposition}

The classes $\left[\,\overline Y_\lambda(d)\right]\in\Z[c_1,c_2]$ can also be expressed in 
\emph{Chern roots} $a$ and $b$: substituting $c_1\mapsto a+b$ and $c_2\mapsto ab$,
we obtain polynomials symmetric in the variables $a$ and $b$.
Writing these symmetric polynomials in the Schur polynomial basis
$s_{|\tilde\lambda|-j,j}=s_{|\tilde\lambda|-j,j}(a,b)$, we get that
\begin{proposition}[{\cite[Prop.~2.5]{feher_juhasz2023plucker}}]
 \label{Y-and-pluecker}
  Let $\lambda$ be a partition without $1$'s. Then
\begin{equation*}
  \left[\,\overline Y_\lambda(d)\right] =\sum_{j=0}^{\lfloor|\tilde\lambda|/2\rfloor}
                 \Pl_{\lambda;|\tilde\lambda|-2j}(d)    s_{|\tilde\lambda|-j,j}.
\end{equation*}
\end{proposition}

\medskip

Theorem \cite[Thm.~2.7]{feher_juhasz2023plucker} is a main novelty of \cite{feher_juhasz2023plucker}:
It provides a new recursive method to calculate and investigate equivariant classes of coincident root
strata.
Here we state a slightly more general version:
\begin{theorem}\label{recursion4Y}
  Let $\lambda=(2^{e_2},\dots,r^{e_r})$ be a nonempty partition without $1$'s and $d\geq |\lambda|$.
  Let $m$ be an element of $\lambda$ and denote by $\lambda'$ the partition $\lambda$ minus $m$,
  $\lambda'=\left(2^{e_2},\dots,m^{e_m-1},\dots,r^{e_r}\right)$.
  We also use the notation  $d'=d-m$. Then
\[ \left[\,\overline Y_\lambda(d)\right]=\frac{1}{e_m}\partial\Big(\left[\,\overline Y_{\lambda'}(d')\right]_{m/d'}\prod_{i=0}^{m-1}\big( ia+(d-i)b \big) \Big),\]
where for a polynomial $\alpha\in\Z[a,b]$ and $q\in \Q$ we use the notation
\[ \alpha_{q}(a,b)=\alpha(a+qa,b+qa)\]
for substituting $a+qa$ and $b+qa$ into the variables $a$ and $b$, and
  \[  \partial(\alpha)(a,b)=\frac{\alpha(a,b)-\alpha(b,a)}{b-a}\]
  denotes the divided difference operation.
\end{theorem}
The only difference between this statement and 
\cite[Thm.~2.7]{feher_juhasz2023plucker} is that in the latter we
chose $m$ to be a maximal element of $\lambda$. This assumption, however, is not necessary and was not
used in the proof.

As $\left[\,\overline Y_{\emptyset}(d)\right]=1$, Theorem \ref{recursion4Y} provides means to investigate the classes $\left[\,\overline Y_\lambda(d)\right]$
using induction on the length of the partitions $\lambda$. For instance, we can prove
an equivalent of Theorem~\ref{thrm_pluckernumberspolynomial}:
\begin{theorem}[{\cite[Thm.~4.1]{feher_juhasz2023plucker}}]
 \label{Y-poly}
 The classes $\left[\,\overline Y_\lambda(d)\right]$ are polynomials in $d$:
 $\left[\,\overline Y_\lambda(d)\right]\in \Q[c_1,c_2,d]$.
\end{theorem}

\medskip

The first interesting cases are the
coincident root strata corresponding to length one partitions $\lambda=(m)$.
Applying Theorem \ref{recursion4Y}, we get that
\[ \left[\,\overline Y_m(d)\right]=
\partial\left(\prod_{i=0}^{m-1}\big( ia+(d-i)b \big) \right) .\]
Analyzing the divided differences $\partial \left( a^ib^{m-i} \right)$, we obtain
\begin{theorem}[{\cite[Thm.~5.1]{feher_juhasz2023plucker}}]\label{m-flex-coeffs}
  For $i$ such that $m-1-i \geq i \geq 0$
\begin{multline*}
  \text{the coefficient of } d^{m-k}s_{m-1-i,i} \text{ in }\left[\,\overline Y_m(d)\right] \text{ is}\\
\begin{cases}
  (-1)^{k+i}\dbinom{k}{i}\stir{m}{m-k}  & \text{if }i\leq k<m-i,
  \\
  \left( (-1)^{k+i}\dbinom{k}{i}- (-1)^{k+m-i} \dbinom{k}{m-i}\right)\stir{m}{m-k} & \text{if }m-i\leq k < m,
  \\
  0 & \text{otherwise,}
\end{cases}
\end{multline*}
where 
\[ \stir{m}{m-k}=\sigma_k(1,2,\dots,m-1)  \]
denotes the Stirling number of the first kind, defined e.g. using the $k$-th elementary symmetric polynomial
$\sigma_k$.
\end{theorem}
In particular, we see that 
 \begin{equation}\label{eq_YmSchurcoeffleadingterm}
   \text{the leading term of the coefficient of } s_{m-1-i,i} \text{ in }\left[\,\overline Y_m(d)\right] =
   \stir{m}{m-i} d^{m-i}.
 \end{equation}

\medskip

Another consequence of Theorem \ref{recursion4Y} shows that
\begin{theorem}[{\cite[Thm.~4.5]{feher_juhasz2023plucker}}]
  \label{fundclass-leading}
  For any $\lambda=\left( 2 ^{e_2},\cdots,r^{e_r} \right)$, the top $d$-degree part of $\left[\,\overline Y_\lambda(d)\right]$ is
  \[ \frac{1}{\prod_{i=2}^r \left( e_i! \right)} h_{\tilde{\lambda}} \, d^{|\lambda|}, \]
where $h_\nu$ is the \emph{complete symmetric polynomial} corresponding to the partition $\nu=(\nu_1,\dots,\nu_k)$: $h_\nu=\prod h_{\nu_i}$
 with $h_i$  the $i$-th complete symmetric polynomial in $\{a,b\}$.
\end{theorem}
So we see that for any partition $\lambda$
\begin{equation}\label{eq_ddegYlambda}
 \deg_d\left( \left[\,\overline Y_\lambda(d)\right] \right)=|\lambda|.
\end{equation}

The definition of the Kostka numbers,
\[ h_{\tilde{\lambda}}=\sum_{j=0}^{\lfloor|\tilde\lambda|/2\rfloor}
K_{(|\tilde{\lambda}|-j,j),\tilde{\lambda}} s_{(|\tilde{\lambda}|-j,j)} \]
together with the fact that
\[ K_{(|\tilde{\lambda}|-j,j),\tilde{\lambda}}=0 \Longleftrightarrow (|\tilde{\lambda}|-j,j) <\tilde{\lambda}
\left( \Longleftrightarrow |\tilde{\lambda}|-j < \lambda_1-1 \right) \]
immediately implies that
\begin{theorem}[{\cite[Thm.~6.1]{feher_juhasz2023plucker}}]\label{thrm_Kostkaleadingcoeffs}
Let $\lambda=(2^{e_2},\dots,r^{e_r})$ be a nonempty partition without $1$'s and
$j\leq |\tilde{\lambda}|-\lambda_1+1$ a nonnegative integer.
 Then
 \[ \text{the leading term of } \Pl_{\lambda;|\tilde{\lambda}|-2j}(d)=
 \frac{K_{(|\tilde{\lambda}|-j,j),\tilde{\lambda}}}{\prod_{i=2}^r \left( e_i!\right) } d^{|\lambda|}, \]
where the $K_{\mu,\nu}$'s denote Kostka numbers.
\end{theorem}

\medskip

Theorem \ref{thrm_leadingtermPl} strenghtens both \eqref{eq_YmSchurcoeffleadingterm} and
Theorem \ref{thrm_Kostkaleadingcoeffs}: It provides a description for the leading terms of all the
generalized Pl\"{u}cker formulas $\Pl_{\lambda;|\tilde{\lambda}|-2j}(d)$.

\section{Leading terms of generalized Pl\"ucker formulas: the proof}\label{sec_proofdegreesofPlnumbers}
This Section is dedicated to the proof of Theorem \ref{thrm_leadingtermPl}.
Our proof results from a quite technical, but purely algebraic analysis of the recursive formula in Theorem
\ref{recursion4Y}; it contains no further geometric ideas.

To make it more concise, let us use the shorthand
$\rho \vdash k$ for partitions $\rho$ of $k$ with length at most 2.
The projection $\pi_2$ onto the second coordinate identifies partitions $\rho \vdash k$ with
elements of the set $\left\{ 0,\dots, \lfloor k/2 \rfloor \right\}$.
We use this identification to introduce an ordering on
$\left\{ \left. \rho \, \right| \rho \vdash k  \right\}$:
\[ \left( k,0 \right) \leq \left( k-1,1 \right) \leq \dots \leq \left( \left\lceil\frac{k}{2}\right\rceil,
\left\lfloor\frac{k}{2}\right\rfloor \right). \]
Also, we can take the differences of $\pi_2$-projections if we want to express ``distance'' of partitions of
$k$.

\medskip

The proof of Theorem \ref{thrm_leadingtermPl} relies on a statement that directly reflects our recursive
formula:
\begin{theorem}\label{thrm_leadingtermfromproduct}
Let $\lambda=\left( 2^{e_2},\dots,r^{e_r} \right)$ be a partition of length at least two. The class of the
corresponding coincident root stratum can be expressed in Schur polynomials
\[ \left[\, \overline Y_\lambda(d)\right]
= \sum_{\rho \vdash c} r_\rho(d) s_\rho \quad \left(c:=| \tilde{\lambda} |=\codim\left(Y_\lambda \subset \Pol^{d}\left( \mathbb{C}^2\right) \right) \right),\]
where  $r_\rho \in \mathbb{Q}\left[ d \right] $.

Let $m$ be any member of $\lambda$ $(e_m \neq 0)$ and
denote by $\lambda'=\left( 2^{e_2},\dots,m^{e_m-1},\dots,r^{e_r}\right)$
the partition $\lambda$ minus $m$.
\begin{enumerate}[label=\roman*)]
\item\label{item_thrmi}
    Then for any $\rho \vdash c$ the coefficients of $s_\rho$ in
    \[     \left[ \, \overline Y_\lambda(d) \right]
      \text{ and }
      \frac{1}{e_m}\left[\, \overline Y_m(d) \right]
    \left[ \, \overline Y_{\lambda'}(d)\right] \]
    have the same leading term.
  \item\label{item_thrmii}
    If $m$ is such that $m-2 \le c':=|\tilde{\lambda'}|=\codim(Y_{\lambda'}\subset \Pol^{d}\left(
    \mathbb{C}^2\right))$, (e.g. $m=\min(\lambda)$), then for any $\rho \vdash c$ the coefficients of
    $s_\rho$ in
    \[     \left[ \, \overline Y_\lambda(d) \right]
      \text{ and }
      \frac{1}{e_m} p_{(m-1,0)}(d) s_{(m-1,0)}
    \left[ \, \overline Y_{\lambda'}(d) \right] \]
    have the same leading term, where $\left[ \, \overline Y_{m}(d)\right]=
    \sum_{\mu \vdash m-1} p_\mu(d) s_\mu$.
  \end{enumerate}
\end{theorem}

Note that Theorem \ref{recursion4Y} would suggest that we compare 
$\left[ \, \overline Y_{\lambda}(d+m)\right]$  with \\
$1/e_m \left[ \, \overline Y_{m}(d+m)\right] \left[ \, \overline Y_{\lambda'}(d)\right]$.
However, the ``$+m$'' translation doesn't change the leading term, hence its omission from
the above Theorem (and from most of this Section).

We will prove Theorem \ref{thrm_leadingtermfromproduct} together with
the following, slightly reformulated, equivalent version of Theorem
\ref{thrm_leadingtermPl} that better suits the equivariant setting (see Proposition \ref{Y-and-pluecker}):
\begin{theorem} \label{thrm_rrholeadterm}
  Let $\lambda=(2^{e_2},\dots,r^{e_r})$ be a nonempty partition without 1's.
  Let $\lambda_1$ be its biggest element, and denote
  by $c=|\tilde{\lambda}|$ the codimension of the corresponding coincident root stratum
  $Y_\lambda \subset \Pol^{d}\left( \mathbb{C}^2\right)$.
  For the coefficients $r_\rho \in \mathbb{Q}[d]$ in the class $ \left[ \, \overline Y_{\lambda}(d) \right]=
  \sum_{\rho \vdash c} r_\rho(d) s_\rho $ there exists a ``threshold'' 
  \[ \vartheta(\lambda)=
  \left( \max \left(  \lambda_1-1, \left\lceil \frac{c}{2} \right\rceil \right),
\min \left( c-\lambda_1+1, \left\lfloor \frac{c}{2} \right\rfloor \right) \right) \]
in the sense that
  \begin{multline*}
    \text{the leading term of } r_{\rho}(d)=\\
    \frac{1}{\prod_{i=2}^r \left( e_i!\right)}
    \begin{cases} 
       K_{\rho,\tilde{\lambda}}\, d^{|\lambda|} &
      \text{if }  \rho \leq \vartheta(\lambda),
      \\[1em]
       \stir{\lambda_1}{\lambda_1-\left( \pi_2(\rho)-\pi_2(\vartheta(\lambda)) \right)}
       d^{|\lambda|-\left( \pi_2(\rho)-\pi_2(\vartheta(\lambda)) \right)} &
      \text{if }  \rho > \vartheta(\lambda),
    \end{cases}
  \end{multline*}
  where the $K_{\nu,\mu}$'s denote Kostka numbers and the $\stir{m}{m-k}$'s are Stirling numbers of the 
  first kind.
\end{theorem}

In particular, the coefficients $r_\rho(d)$ have positive leading coefficients,
which can also be seen by their interpretation as enumerative problems, see Proposition \ref{Y-and-pluecker}.

The following figure illustrates the degree distribution of the coefficients of Schur polynomials in
a class $\left[ \, \overline Y_{\lambda}(d) \right]$.

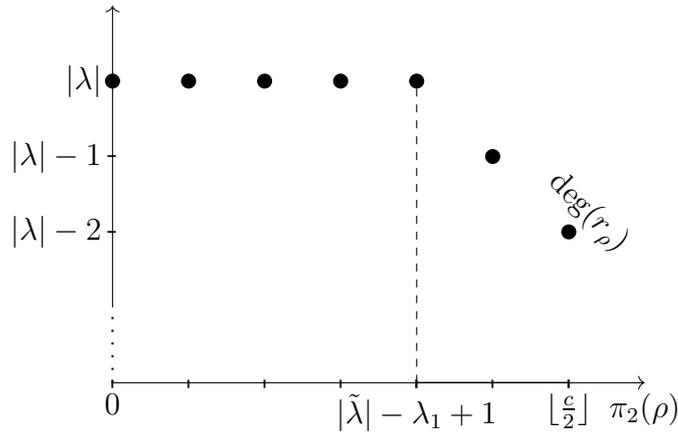
\begin{figure}[H]
  \centering
 \pgfmathsetmacro{\addlambda}{16}
\pgfmathsetmacro{\lambdamax}{10}
\pgfmathsetmacro{\c}{13}
\pgfmathsetmacro{\rhobtill}{floor(\c/2)}
\pgfmathsetmacro{\thetalambda}{\c-\lambdamax+1}

\pgfmathsetmacro{\ycutto}{1}
\pgfmathsetmacro{\ycutdiff}{12}

\begin{tikzpicture}
  \draw[loosely dotted,thick] (0,0) -- (0,\ycutto);
\draw[->]
(0,\ycutto)  --
(0,\addlambda-\ycutdiff-2) node [left] {$|\lambda|-2$} node {-} --
(0,\addlambda-\ycutdiff-1) node [left] {$|\lambda|-1$} node {-} --
(0,\addlambda-\ycutdiff) node [left] {$|\lambda|$} node {-} --
(0,\addlambda-\ycutdiff+1) ;

\draw[->]
(0,0) node [below] {$0$} --
(\rhobtill,0) node [rotate=90] {-} node [below] {$\lfloor \frac{c}{2} \rfloor$} --
(\thetalambda,0) node [rotate=90] {-} node [below] {$|\tilde{\lambda}|-\lambda_1+1$} --
(\rhobtill+1,0) node [below] {$\pi_2(\rho)$};
\foreach \rhob in {0,...,\rhobtill}
{
\draw (\rhob,0) node [rotate=90] {-} ;
}

\foreach \rhob in {0,...,\rhobtill}
{
\pgfmathsetmacro{\degrrho}{min(\addlambda,\addlambda+\thetalambda-\rhob)}
\draw (\rhob,\degrrho-\ycutdiff) node [circle,fill,inner sep=2pt] {};
}

\draw[dashed] (\thetalambda,0) -- (\thetalambda,\addlambda-\ycutdiff);

\pgfmathsetmacro{\degrrholabelb}{\addlambda-\ycutdiff-(\rhobtill-\thetalambda)}
\draw (\rhobtill,\degrrholabelb) node [above,rotate=-45] {$\deg(r_\rho)$};

\end{tikzpicture}
 \caption{Degrees of coefficients $r_\rho(d)$ in $\left[ \, \overline Y_{\lambda}(d) \right]=
 \sum_{\rho \vdash c} r_\rho(d) s_\rho $ for $\lambda=(10,3,3)$}
  \label{tikz_degreedrop}
\end{figure}

\begin{proof}
  We will prove Theorem \ref{thrm_leadingtermfromproduct} and Theorem \ref{thrm_rrholeadterm} simultaneously
  using induction on
  the length of the partition $\lambda$.
  Throughout the proof we will keep on using the following shorthands for codimensions:
  \[ c=|\tilde{\lambda}|=\codim(Y_\lambda), \quad c'=|\tilde{\lambda'}|=\codim(Y_{\lambda'}),
  \quad m'=m-1=\codim(Y_m). \]

  Induction starts with $\lambda=(m)$, where Theorem \ref{thrm_leadingtermfromproduct} is empty,
  $\vartheta(m)=m'-m+1=0$, and, by \eqref{eq_YmSchurcoeffleadingterm},
  the coefficients of the Schur polynomials in
\[\left[ \, \overline Y_{m}(d)\right]=\sum_{\mu \vdash m'} p_{\mu}(d) s_{\mu} \]
have the expected leading term, proving Theorem \ref{thrm_rrholeadterm}.

The induction step is based on the recursion of Theorem \ref{recursion4Y}.
Accordingly,
choose an element $m$ of $\lambda$, and let $\lambda'=\left( 2^{e_2},\dots,m^{e_m-1},\dots,r^{e_r}\right)$
be as in Theorem \ref{recursion4Y} (or \ref{thrm_leadingtermfromproduct}).
The partition $\lambda'$ has length one less than $\lambda$, so
we can assume that Theorem \ref{thrm_rrholeadterm}
holds for the coefficients $q_{\nu}(d)$ in
\[ \left[\, \overline Y_{\lambda'}(d)\right]
= \sum_{\nu \vdash c'} q_\nu(d) s_\nu .  \]

The substitutions $a \mapsto a + (m/d)a$ and $b \mapsto b + (m/d)a$ in the recursive formula of Theorem \ref{recursion4Y} can be divided into two steps:
\begin{equation}\label{eq_CRSrecursion}
  \begin{split}
     \left[ \, \overline Y_\lambda(d+m) \right]&=
     \frac{1}{e_m} \partial \left(
       \left[ \, \overline Y_{\lambda'}(d) \right]
       \big\rvert_{\stackon{$\hspace{.075em} \scriptstyle b \mapsto b + (m/d)a$}{$\scriptstyle a  \mapsto  a+(m/d)a$}}
     \cdot \prod_{i=0}^{m-1} \left( ia +(d+m-i)b \right) \right)\\
     &=\frac{1}{e_m} \partial \left(
       \left[ \, \overline Y_{\lambda'}(d) \right]
       \big\rvert_{\stackon{$\hspace{.075em} \scriptstyle b \mapsto b+x$}{$\scriptstyle a \mapsto a+x$}} \bigg\vert_{x
       \mapsto (m/d)a}
     \cdot \prod_{i=0}^{m-1} \left( ia +(d+m-i)b \right) \right).
  \end{split}
\end{equation}
Let us keep the variable $x$ for a moment, and define $B_t$ ($t=0,\dots,c'$) as the
coefficient of $x^t$ in
\begin{equation}\label{eq_Bt_def}
\left[ \, \overline Y_{\lambda'}(d) \right]
  \big\rvert_{\stackon{$\hspace{.075em} \scriptstyle b \mapsto b+x$}{$\scriptstyle a \mapsto a+x$}}=
\sum_{t=0}^{c'} B_t x^t.
\end{equation}
The polynomials $B_t \in \mathbb{Q}[a,b;d]^{S_2}$ are symmetric in $a,b$ and have $\{a,b\}$-degree $c'-t$. 
Note that $B_0=\left[ \, \overline Y_{\lambda'}(d) \right]$.

We can expand \eqref{eq_CRSrecursion} as
\begin{equation}\label{eq_Ylambda_texpansion}
  \begin{split}
     \left[ \, \overline Y_\lambda(d+m) \right]&=
     \frac{1}{e_m} \partial \left( \sum_{t=0}^{c'} \left( B_t \left( \frac{m}{d}a \right)^t \right)
     \cdot \prod_{i=0}^{m-1} \left( ia +(d+m-i)b \right) \right)
     \\
     &=\frac{1}{e_m} \sum_{t=0}^{c'} B_t \left( \frac{m}{d} \right)^t \cdot \partial\left( a^t  \prod_{i=0}^{m-1}
     \left( ia +(d+m-i)b \right)\right)=
     \frac{1}{e_m} \sum_{t=0}^{c'} \left( \frac{m}{d} \right)^t A_t B_t ,
  \end{split}
\end{equation}
where we denoted by $A_t$ ($t=0,\dots,c'$) the divided differences
\[ A_t=\partial\left( a^t  \prod_{i=0}^{m-1}\left( ia +(d+m-i)b \right)\right). \]
The polynomials $A_t \in \mathbb{Q}[a,b;d]^{S_2}$ are symmetric in $a,b$ and have $\{a,b\}$-degree $m'+t$.
Note that $A_0=\left[\,\overline Y_m(d+m)\right]$.

Introducing coefficients $p_\mu \in \mathbb{Z}[d]$ ($\mu \vdash m'+t$) and $q_\nu \in \mathbb{Q}[d]$
($\nu \vdash c'-t$) of $A_t$ and $B_t$
in the Schur polynomial basis,
\begin{equation}\label{eq_pmuqnudef}
A_t=\sum_{\mu \vdash m'+t} p_\mu s_\mu \quad \text{ and } \quad
B_t=\sum_{\nu \vdash c'-t} q_\nu s_\nu, 
\end{equation}
we can continue \eqref{eq_Ylambda_texpansion} as
\begin{equation}\label{eq_Ylambda_tfullexpansion}
    \left[ \, \overline Y_\lambda(d+m) \right]=\frac{1}{e_m} \sum_{t=0}^{c'} \left( \frac{m}{d} \right)^t
    A_t B_t 
    =\frac{1}{e_m} \sum_{t=0}^{c'} \left( \frac{m}{d} \right)^t
    \sum_{\mu \vdash m'+t} \left( p_\mu s_\mu \right) \sum_{\nu \vdash c'-t} \left( q_\nu s_\nu \right).
 \end{equation}

Let us add here that the terms $(m/d)^t$ made us think that for the higher $t$'s the $d$-degrees of the
corresponding summands in \eqref{eq_Ylambda_tfullexpansion} might be lower.

Figure~\ref{tikz_AtBtoverview} is meant to depict \eqref{eq_Ylambda_tfullexpansion}:
for each $t=0,\dots,c'$ the left-hand side of its $t$-th row consists of partitions
$\mu \vdash m'+t$ representing the terms $p_\mu s_\mu$ of $A_t$ and the right-hand side of its $t$-th row
comprises partitions $\nu \vdash c'-t$ representing the terms $q_\nu s_\nu$ of $B_t$.
Line segments of Figure~\ref{tikz_AtBtoverview} we will explain later.

\begin{figure}
  \centering
 \pgfmathsetmacro{\m}{5}
\pgfmathsetmacro{\cprime}{8}

\begin{tikzpicture}[xscale=1.1]
  \pgfpointtransformed{\pgfpointxy{1}{1}};
  \pgfgetlastxy{\vx}{\vy}
  \begin{scope}[node distance=\vy and \vx] 
  \foreach \t in {0,...,\cprime}
  \pgfmathsetmacro{\z}{(\t+\m-1)/2}
  \foreach \x in  {0,...,\z}
  {
    \pgfmathparse{int(\t+\m-1-\x)}
    \let\theIntINeed\pgfmathresult
    \draw (\x,-\t) node [draw,rounded corners,inner sep=2pt] {$\theIntINeed,\x$};
  };

  \pgfmathsetmacro{\posseg}{(\m-1)/2}
  \foreach \pos in {0,...,\posseg}
  {
    \pgfmathsetmacro{\endx}{\m-1-\pos};
    \pgfmathsetmacro{\endy}{2*\pos-\m+1};
    \pgfmathsetmacro{\posseglabel}{int(\endx+1)};
    \draw[line cap=round,color=cyan,line width=3pt,draw opacity=0.5] (\pos-0.5,0+0.5) --
    node[above,at start,xshift=-8,yshift=-6] {$\posseglabel$} (\endx+0.5,\endy-0.5);
  }

  \pgfmathsetmacro{\negseg}{\m}
  \foreach \neg in {1,...,\negseg}
  {
    \pgfmathparse{\m-1-(\neg+max(\neg,\m-1-\neg))};
    \let\startnegy\pgfmathresult;
    \draw[line cap=round,color=red,line width=3pt,draw opacity=0.7] (\neg,\startnegy+.5) --
    node [above, at start] {$\neg$}
    (\neg,-\cprime-.5);
  }

  \pgfmathparse{int(floor((\m-1+\cprime)/2)+2)};
  \let\RHSx\pgfmathresult;

  \foreach \t in {0,...,\cprime}
  \pgfmathsetmacro{\z}{(\cprime-\t)/2}
  \foreach \Schtwo in  {0,...,\z}
  {
    \pgfmathparse{int(\cprime-\t-\Schtwo)};
    \let\Schone\pgfmathresult
    \draw (\RHSx+\Schtwo,-\t) node [draw,rounded corners,inner sep=2pt]  {$\Schone,\Schtwo$};
  };


  \pgfmathsetmacro{\dashedx}{int(floor((\m-1+\cprime)/2)+1)}
  \draw[dashed,thick] (\dashedx,.5) -- (\dashedx,-\cprime-.5);
  
\end{scope}
\end{tikzpicture}
 \caption{Overview of the products $A_t B_t$ ($t=0,\dots,c'$) for $m=5$ and $c'=8$}
  \label{tikz_AtBtoverview}
\end{figure}
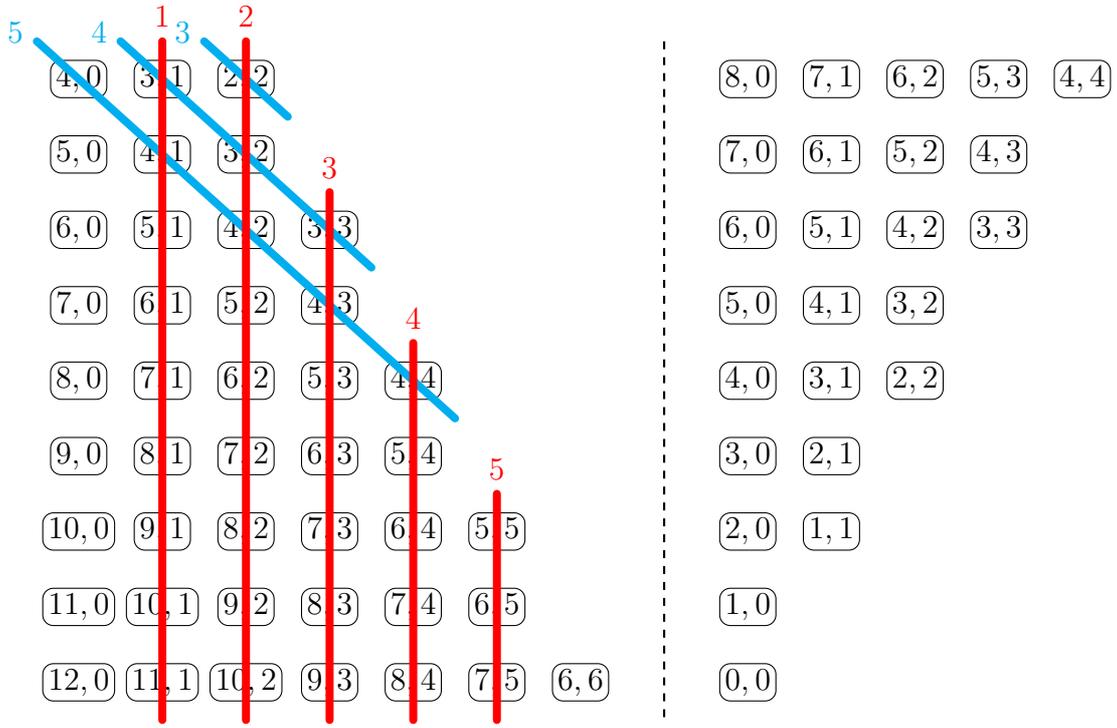

\medskip

Theorem \ref{thrm_leadingtermfromproduct}/\ref{item_thrmi} follows from 
\begin{multline}\label{eq_thrm1intermsoft}\tag{$*$}
  \deg\left( \text{coefficient of } s_\rho \text{ in }  A_0 B_0 \right) >
  \deg\left( \text{coefficient of } s_\rho \text{ in } \left( \frac{m}{d} \right)^{t} A_t B_t\right)
  \\
  \text{holds for every $\rho \vdash c=c'+m'$ and every $t=1,\dots,c'$.}
\end{multline}
The proof of (\ref{eq_thrm1intermsoft}) will take up the majority of what follows (and will end in part
\ref{item_puttingtogether}, see later).
Part \ref{item_thrmii} of Theorem \ref{thrm_leadingtermfromproduct} will result from a further analysis
of the $t=0$ summand of \eqref{eq_Ylambda_tfullexpansion}.
Finally, Theorem \ref{thrm_rrholeadterm} will be proved by choosing $m=\min(\lambda)$.

\medskip

The line segments of Figure \ref{tikz_AtBtoverview}
will be used to prove (\ref{eq_thrm1intermsoft}):
these segments will be defined such that they cover all the partitions with corresponding coefficients
$p_\mu \in \mathbb{Q}[d]$ 
nonzero and such that along them the behaviour of $\deg(p_\mu)$ 
and the sign of the leading coefficient of $p_\mu$ 
can be studied.

The line segments of the left-hand side
will also be used to compare the summands
\begin{equation}\label{eq_pmuterm}
  \left( \frac{m}{d} \right)^t p_\mu s_\mu B_t=:\sum_{\rho \vdash c} r_{\mu,\rho}(d) s_\rho
\end{equation}
($\mu \vdash m'+t$) of $(m/d)^t A_t B_t$. More precisely, 
given any partition $\rho \vdash c$ we will compare their coefficients $r_{\mu,\rho}(d)
\left( \in \mathbb{Q}[d] \text{ , see Proposition \ref{prop_rmurho_ispoly}} \right)$
along these vertical and diagonal line segments.
Based on these comparisons,
we will be able to compare for different $t$'s the $d$-degrees of 
\[ \text{the coefficients of } s_\rho \text{ in the terms } \left( \frac{m}{d} \right)^t A_t B_t =
\sum_{\mu \vdash m'+t} r_{\mu,\rho}(d), \]
which will eventually lead to the proof of (\ref{eq_thrm1intermsoft}).

\pagebreak

\noindent
To accomplish the plan outlined above, we proceed with the following steps.
\begin{enumerate}
  \item[\ref{item_degq}]
    We show that $\deg(q_\nu)$ depends only on $\pi_2(\nu)$
    and that its leading coefficient is always positive.
  \item[\ref{item_degsignp}]
    We define diagonal and vertical line segments of $\big\{ \mu \vdash m'+t | \, 0 \le t \le c' \big\}$, and
    show how the degree and the sign of the leading coefficient of $p_\mu(d)$ can be deduced from
    the line segment(s) $\mu$ is contained in.
  \item[\ref{item_degformu}]
    For every $\rho \vdash c$ we define functions
    \[ f_\rho: \left\{ \left. \mu  \right| \mu \vdash m'+t \right\} \to
	\left\{ \left. \nu  \right| \nu \vdash c'-t \right\} \cup \left\{ \infty \right\} \]
	($t=0,\dots,c'$) that will help us to determine degrees of the coefficients $r_{\mu,\rho}(d)$.
  \item[\ref{item_comparefrhoalongsegments}]
    For any given $\rho \vdash c$ we compare values $\pi_2 \left( f_\rho(\mu)\right)$ for adjacent partitions $\mu$ of
    diagonal and vertical line segments.
  \item[\ref{item_puttingtogether}]
    By connecting any $\mu \vdash m'+t$ ($t \ge 1$ and $p_\mu \neq 0$) to the $t=0$ row via line segments
    and making the above comparisons along the way, we prove \eqref{eq_thrm1intermsoft}.
  \item[\ref{item_thrm1ii}] We prove the \ref{item_thrmii} case of Theorem \ref{thrm_leadingtermfromproduct}.
\item[\ref{item_thrm2}] We conclude with a proof for Theorem \ref{thrm_rrholeadterm}.
\end{enumerate}

\medskip

\begin{enumerate}[label = {\textbf{(\Alph*)}},wide, labelindent=0pt]
  \item\label{item_degq}
    Let us start by investigating the coefficients $q_{\nu}(d)$ in the $B_t$'s.
    A simple substitution into Jacobi's bialternant formula shows that
\begin{equation}\label{eq_subs2Schur}
  s_{(k,l)}\big\rvert_{\stackon{$\hspace{.075em} \scriptstyle b \mapsto b+x$}{$\scriptstyle a \mapsto a+x$}}=
  \mathlarger{\sum}_{t=0}^{k+l} x^{c'-t} \mathlarger{\sum}_{(u,v)\vdash t}
  \left( \binom{k+1}{u+1} \binom{l}{v}-\binom{k+1}{v}\binom{l}{u+1} \right) s_{(u,v)}.
\end{equation}
Here, for all the $s_{(u,v)}$'s their coefficients are nonnegative 
and zero if $u>k$ or $v>l$. Hence, the coefficient of $s_{(u,v)}$ ($(u,v) \vdash c'-t$) in
\[ \left[ \, \overline Y_{\lambda'}(d) \right]
  \big\rvert_{\stackon{$\hspace{.075em} \scriptstyle b \mapsto b+x$}{$\scriptstyle a \mapsto a+x$}}=
\sum_{\nu \vdash c'} q_{\nu} s_{\nu}\big\rvert_{\stackon{$\hspace{.075em} \scriptstyle b \mapsto b+x$}{$\scriptstyle a \mapsto a+x$}}\]
is $x^t$ times a linear combination of elements in
$\left\{ \left. q_{(k,l)} \right| (k,l) \vdash c', u \le k \text{ and } v \le l \right\}$
with positive coefficients.

Using the positivity and the monotone decreasing nature of
$\left\{\deg (q_{\nu})\right\}_{\nu \vdash c'}$ as in 
Theorem \ref{thrm_rrholeadterm} part of the induction hypothesis for $\lambda'$,
we deduce that for every $t=0,\dots,c'$ and $(u,v) \vdash c'-t$
\begin{equation}\label{eq_degq}
  \deg \left( q_{(u,v)}\right) =\deg \left(  q_{(u+t,v)}\right) \text{ and the leading coefficient of } q_{(u,v)} \text{ is positive}.
\end{equation}
In other words, $\deg \left( q_\nu\right)$ depends only on $\pi_2(\nu)$. Therefore,
for any given $t=0,\dots,c'$ with respect to our ordering of partitions $\nu \vdash c'-t$
\begin{equation}\label{eq_degqmondecreasing}
  \begin{gathered}
    \deg \left( q_\nu \right) \text{ is monotone decreasing  and} \\
    \text{ it's difference for adjacent } \nu\text{'s is at most }  1.
\end{gathered}
\end{equation}

\medskip

\item\label{item_degsignp}
Expanding its definition, we can write $A_t$ as
\[ A_t=\partial\left( a^t  \prod_{i=0}^{m-1}\left( ia +(d+m-i)b \right)\right)=
\sum_{f=1}^{m} e_{f}(d) \partial\left( a^{t+m-f}b^{f} \right), \]
for the $e_f \in \mathbb{Z}[d]$ coefficients in $\prod_{i=0}^{m-1}(ia+(d+m-i)b)=\sum_{f=1}^m e_f(d) a^{m-f} b^f$.
In particular,
$\deg(e_f)=f$ and its leading coefficient is positive.

Using
\begin{equation*}
  \threecase{\partial\left( a^{m+t-f}b^{f}\right)=}
  {s_{(f-1,m+t-f)}}{2f>m+t,}
  {0}{2f=m+t,}
  {-s_{(m+t-1-f,f)}}{2f<m+t,}
\end{equation*}
we get that for any $f \in \{1,\dots,m\}$
$+e_f$ is a summand of the coefficient of $s_{(f-1,m-f+t)}$ in $A_t$ for $t<2f-m$ and 
$-e_f$ is a summand of the coefficient of $s_{(m+t-1-f,f)}$ in $A_t$ for $t>2f-m$, see Figure
\ref{tikz_Atsegments_why}. For example, in the $m=5$ and $t=1$ case we have $p_{(4,1)}=e_5-e_1$.
 \pgfmathsetmacro{\m}{5}
\pgfmathsetmacro{\cprime}{8}

\newcommand\nodecomparea{
     \ifthenelse{\lhs>\rhs}{\nodetextaa}{}
     \ifthenelse{\equal{\lhs}{\rhs}}{\nodetextba}{}
     \ifthenelse{\lhs<\rhs}{\nodetextca}{}
   }
   \newcommand\nodecompareb{
     \ifthenelse{\lhs>\rhs}{\nodetextab}{}
     \ifthenelse{\equal{\lhs}{\rhs}}{\nodetextbb}{}
     \ifthenelse{\lhs<\rhs}{\nodetextcb}{}
   }
   \newcommand\nodesign{
     \ifthenelse{\lhs>\rhs}{+}{}
     \ifthenelse{\equal{\lhs}{\rhs}}{}{}
     \ifthenelse{\lhs<\rhs}{-}{}
   }

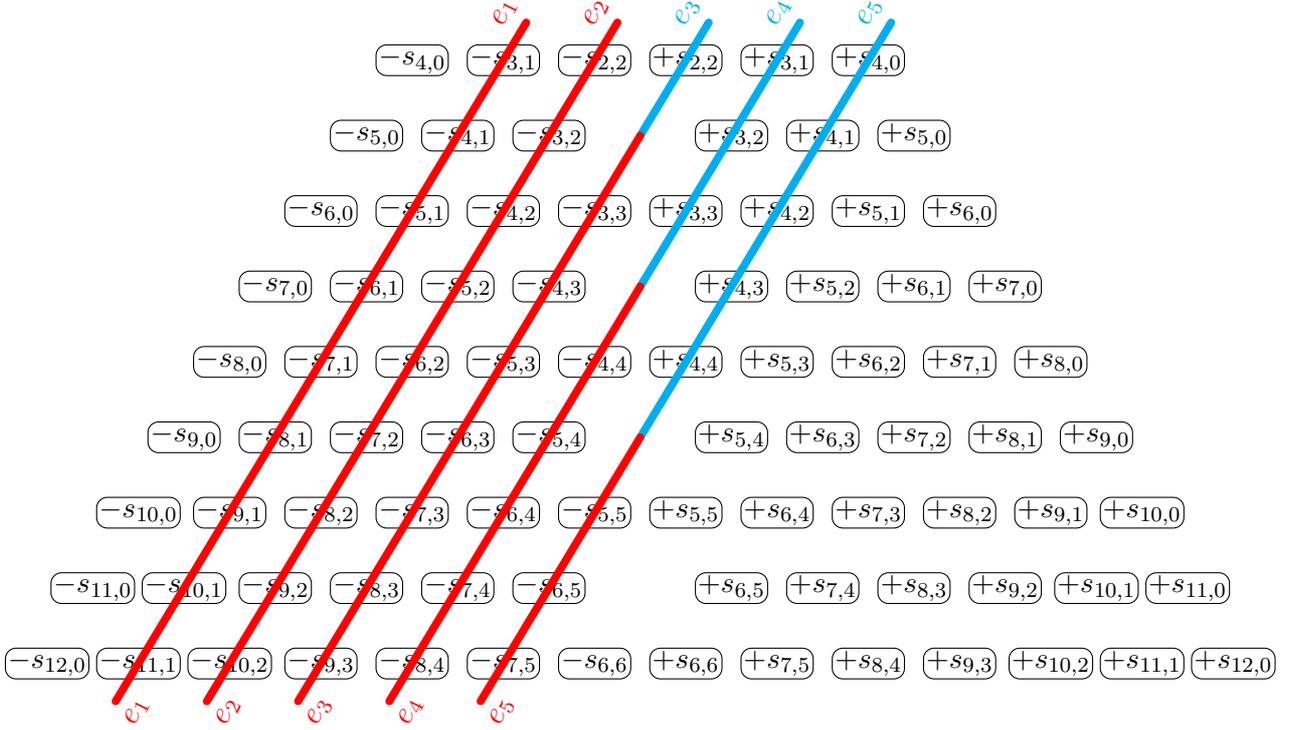
\begin{figure}
  \centering
  \begin{tikzpicture}[xscale=1.2]
    \pgfpointtransformed{\pgfpointxy{1}{1}};
    \pgfgetlastxy{\vx}{\vy}
    \begin{scope}[node distance=\vy and \vx] 
    \foreach \t in {0,...,\cprime}
    \pgfmathsetmacro{\ftill}{\m+\t}
    \foreach \f in  {0,...,\ftill}
    {
      \pgfmathsetmacro{\tcorrection}{(\m+\t)/2}
      \pgfmathsetmacro{\lhs}{int(2*\f)}
      \pgfmathsetmacro{\rhs}{int(\m+\t)}
      \pgfmathparse{int(\f-1)}
      \let\nodetextaa\pgfmathresult
      \pgfmathparse{int(\m+\t-\f)}
      \let\nodetextab\pgfmathresult
      \pgfmathparse{int(-10)}
      \let\nodetextba\pgfmathresult
      \pgfmathparse{int(-20)}
      \let\nodetextbb\pgfmathresult
      \pgfmathparse{int(\m+\t-1-\f)}
      \let\nodetextca\pgfmathresult
      \pgfmathparse{int(\f)}
      \let\nodetextcb\pgfmathresult
      \pgfmathparse{int(\t+\m-\f)}
      \let\aexponent\pgfmathresult
      \ifthenelse{\NOT \equal{\lhs}{\rhs}}{
	\draw (\f-\tcorrection,-\t) node [font=\small,draw,rounded corners,inner sep=1pt]
	{$\nodesign s_{\nodecomparea,\nodecompareb}$}
      }
      {};
    };
    
    \foreach \f in {1,...,\m}
    \pgfmathsetmacro{\lhs}{int(2*\f)}
    \pgfmathsetmacro{\rhs}{int(\m)}
    \ifthenelse{\lhs<\rhs \OR \lhs=\rhs}
    {
      \pgfmathsetmacro{\startx}{\f-\m/2}
      \pgfmathsetmacro{\starty}{0}
      \pgfmathsetmacro{\endx}{\f-(\m+\cprime)/2}
      \pgfmathsetmacro{\endy}{-\cprime}
      \draw[line width=3pt, line cap=round,color=red,draw opacity=0.7]  (\startx+.25,\starty+.5)
      to node [sloped, above, at start] {$e_{\f}$}
      node [sloped, below, at end] {$e_{\f}$}
      (\endx-.25,\endy-.5)
    }
    {
      \pgfmathsetmacro{\startx}{\f-\m/2}
      \pgfmathsetmacro{\starty}{0}
      \pgfmathsetmacro{\midx}{0}
      \pgfmathsetmacro{\midy}{-(2*\f-\m)}
      \pgfmathsetmacro{\endx}{\f-(\m+\cprime)/2}
      \pgfmathsetmacro{\endy}{-\cprime}
      \draw[line width=3pt, line cap=round,color=cyan,draw opacity=0.5]  (\startx+.25,\starty+.5)
      to node [sloped, above, at start] {$e_{\f}$} (\midx,\midy);
      \draw[line width=3pt, line cap=round,color=red,draw opacity=0.7]  (\midx,\midy)
      to node [sloped, below, at end] {$e_{\f}$} (\endx-.25,\endy-.5)
    };
  \end{scope}
  \end{tikzpicture}
  \caption{The $A_t$'s as linear combinations of the divided differences
  $\partial(a^{t+m-f}b^{f})=\pm s_\mu$ with nonzero coefficients indicated by the line segments $e_f$
in the $m=5$ and $c'=8$ case}
  \label{tikz_Atsegments_why}
\end{figure}

Accordingly, for each $(i,j) \vdash m'$ the set 
\[ \left\{ \left. (i,j+t) \vdash m'+t \right| 0 \le t \le i-j \right\}\]
  ---where the corresponding coefficients $p_{(i,j+t)}$ have a $+e_{i+1}$ summand---we
  will call the  \emph{$(i+1)$-diagonal line segment},
  and for each
  $j \in \left\{ 1,\dots,m \right\}$ the set
\[ \left\{ \left. (i,j)   \right| j \le i \le m'+c'-j \right\} \]
  ---where the corresponding coefficients $p_{(i,j)}$ have a $-e_{j}$ summand---we
  will call the \emph{$j$-vertical line segment}, see Figure \ref{tikz_AtBtoverview}.

\medskip

\item\label{item_degformu}
Products of Schur polynomials in two variables can be easily calculated using e.g. Pieri's formula.
For $\mu \vdash m'+t$ and $\nu \vdash c'-t$
\[ s_{\mu} s_{\nu}=\sum_{\rho \in I(\mu,\nu)} s_\rho, \]
where $I(\mu,\nu)=\left[ p(\mu,\nu),P(\mu,\nu) \right]$ is an interval of partitions $\rho \vdash c$ with
endpoints
\[ p\left( (i,j),(k,l) \right)=(i+k,j+l) \enskip \text{ and } \enskip P\left( (i,j),(k,l) \right)=
\left(\max(i+l,j+k),\min(i+l,j+k)\right).\]

Let us note here that for any $\mu \vdash m'+t$ and adjacent partitions $(k,l), (k-1,l+1) \vdash c'-t$
the fact that
both the starting points and the endpoints of $I(\mu,(k,l))$ and $I(\mu,(k-1,l+1))$ can differ by at
most one implies that
\begin{equation}\label{eq_unionalsointerval}
  \text{for any }\nu_1, \nu_2 \vdash c'-t \text{ the union } \bigcup_{\nu_1 \le \nu \le \nu_2} I(\mu,\nu) \text{ is also an interval.}
\end{equation}

Using the notation $I(\mu,\nu)$, we can express the coefficient $r_{\mu,\rho}$ $(\mu \vdash m'+t)$ defined
in \eqref{eq_pmuterm} as 
\begin{equation}\label{eq_rmurho_equals}
  r_{\mu,\rho}(d)=\left( \frac{m}{d} \right)^t \sum_{\substack{\nu \vdash c'-t \\ \rho \in I(\mu,\nu)}}
  p_{\mu}(d) q_{\nu}(d).
\end{equation}

The positivity of the leading coefficients of the $q_{\nu}$'s, see (\ref{eq_degq}), implies that if
$\rho \vdash c$ and $\mu \vdash m'+t$ are partitions such that $\rho \in \cup_{\nu \vdash c'-t} I(\mu,\nu)$,
then $r_{\mu,\rho} \neq 0$.
In Appendix \ref{sec_polynomiality_rmurho} we show that the a priori rational functions $r_{\mu,\rho}(d)$'s
are polynomials in $d$. Although not strictly necessary for our proof to work, this makes the interpretation
of their degrees unambiguous.
Our goal is to analyze this $\deg(r_{\mu,\rho})$. 

As, by \eqref{eq_degqmondecreasing}, $\deg (q_{\nu})$ is monotone decreasing in $\nu \vdash c'-t$ ,
we are interested in the smallest $\nu \vdash c'-t$ such that $\rho \in I(\mu,\nu)$.
Therefore, for every $\rho \vdash c$ and $\mu \vdash m'+t$ we define
\[  \twocase{f_\rho\left( \mu \right)=}
  {\min\left\{ \left. \nu \vdash c'-t \right| \rho\in I\left( \mu,\nu \right)   \right\}}
    {\rho \in \bigcup_{\nu \vdash c'-t} I\left( \mu,\nu \right),}
  {\infty}{\rho \notin \bigcup_{\nu \vdash c'-t} I\left( \mu,\nu \right).}\]
This function will be crucial in our proof as,
again by the positivity of the leading coefficients of the $q_\nu$'s,
\[ \deg\left( r_{\mu,\rho} \right)=
\deg\left( \text{the coefficient of } s_\rho \text{ in } \left( \frac{m}{d} \right)^t p_\mu s_\mu q_{f_{\rho}(\mu)} s_{f_{\rho}(\mu)} \right), \]
where in case $\rho \notin \bigcup_{\nu \vdash c'-t} I\left( \mu,\nu \right)$ and $f_\rho(\mu)=\infty$,
we set $q_\infty =s_\infty =0$.
In other words,
\begin{equation}\label{eq_degformu}
     \deg\left( r_{\mu,\rho} \right)=
     \deg \left( p_\mu \right)+
     \deg\left( q_{f_\rho(\mu)} \right) - t.
   \end{equation}
Combining this with \eqref{eq_degqmondecreasing}, we get that
if $\mu_i \vdash m+t_i$ and $\rho \vdash c$ are partitions such that
$\pi_2(f_{\rho}(\mu_1)) \le \pi_2(f_{\rho}(\mu_2))+f$ for some $f \in \mathbb{N}_0$, then
$\deg(q_{f_\rho(\mu_1)})+f \ge \deg(q_{f_\rho(\mu_2)})$, hence
\begin{multline}\label{eq_degsrhocoeff}
  \pi_2(f_{\rho}(\mu_1)) \le \pi_2(f_{\rho}(\mu_2))+f  \implies
  \\[5pt]
  \deg(r_{\mu_1,\rho}) \ge
  \deg(r_{\mu_2,\rho})+\deg(p_{\mu_1})-\deg(p_{\mu_2})+t_2-t_1-f.
\end{multline}
If we set $\pi_2(\infty)=\infty$,
and define the degree of the constant 0 polynomial to be $-\infty$,
\eqref{eq_degformu} and \eqref{eq_degsrhocoeff} remains valid even when
$\rho \notin \bigcup_{\nu \vdash c'-t_{(2)}} I\left( \mu_{(2)},\nu \right)$,
allowing a uniform treatment of all the cases.

\medskip

\item\label{item_comparefrhoalongsegments}
The goal of this part is to compare the values $\pi_2(f_\rho(\mu))$ along
diagonal (\ref{item_mon4diagonal}) and vertical (\ref{item_mon4horizontal}) line segments:

\smallskip

\begin{enumerate}[label = (D/\Roman*),wide, labelindent=0pt]
  \item\label{item_mon4diagonal}
Let $\mu_1=(i,j) \vdash m'+t$ and $\mu_2=(i,j+1) \vdash m'+(t+1)$ be adjacent partitions of a diagonal segment.
Then for every $(k,l) \vdash c'-(t+1)$
\[ p\left( (i,j),(k+1,l) \right)<p\left( (i,j+1),(k,l) \right) \text{ and }
P\left( (i,j),(k+1,l) \right)=P\left( (i,j+1),(k,l) \right), \]
hence
\begin{equation}\label{eq_intervalcontained4diag}
I\left( (i,j),(k+1,l) \right) \supset I\left( (i,j+1),(k,l) \right),
\end{equation}
which in turn---as illustrated by Figure \ref{tikz_diagonalsegmentintervals} with an example---implies that for
every $\rho \vdash c$
\begin{equation}\label{eq_pi2frhocomparison_diagonal}
\pi_2(f_\rho(\mu_1)) \le \pi_2(f_\rho(\mu_2)).
\end{equation}

\pgfmathsetmacro{\m}{5}
\pgfmathsetmacro{\cprime}{8}
\pgfmathsetmacro{\c}{\m-1+\cprime}
\pgfmathsetmacro{\rhotill}{\c/2}
    \pgfmathsetmacro{\rhoplace}{\rhotill/2}
\pgfmathsetmacro{\mua}{4}
\pgfmathsetmacro{\mub}{2}
\pgfmathsetmacro{\t}{2}

\begin{figure}
  \centering
  \begin{tikzpicture}[xscale=0.9, every node/.style={font=\footnotesize}]
  \pgfpointtransformed{\pgfpointxy{1}{1}};
  \pgfgetlastxy{\vx}{\vy}
  \begin{scope}[node distance=\vy and \vx] 

    \pgfmathsetmacro{\seplinex}{-0.55}
    \pgfmathsetmacro{\nubtill}{floor((\cprime-\t)/2)}
    \draw (-1.2,.5) node [name=mu] {$(k\!\!+\!\!1,l)$};
    \draw (\rhoplace,.5) node [name=rho] {$\rho$};

    \draw[thin] (\seplinex,.7) -- (\seplinex,-\nubtill-0.25);
    \draw[thin] (-1.8,0.3) -- (\rhotill+.5,0.3);

    \foreach \nub in {0,...,\nubtill}
    {
    \pgfmathsetmacro{\nua}{\cprime-\t-\nub}
    \pgfmathsetmacro{\nuadescr}{int(\nua)}
    \draw (-1,-\nub) node {$(\nuadescr,\nub)$};

    \foreach \rhob in {0,...,\rhotill}
    {
      \pgfmathsetmacro{\rhoa}{int(\c-\rhob)}
      \pgfmathsetmacro{\rhobdescr}{int(\rhob)}
      \draw (\rhob,-\nub) node [draw,rounded corners,inner sep=2pt] {$\rhoa,\rhobdescr$};
    }
    \pgfmathsetmacro{\rhobstart}{\mub+\nub}
    \pgfmathparse{min(\mua+\nub,\mub+\nua)}
    \let\rhobend\pgfmathresult
    \foreach \rhob in {\rhobstart,...,\rhobend}
      \pgfmathsetmacro{\rhoa}{int(\c-\rhob)}
      \pgfmathsetmacro{\rhobdescr}{int(\rhob)}
      \draw (\rhob,-\nub) node [draw,rounded corners,fill=black!20,inner sep=2pt] {$\rhoa,\rhobdescr$};
  }

    \pgfmathsetmacro{\dashedx}{floor(\c/2)+1}
    \draw[dashed,thick] (\dashedx,.5) -- (\dashedx,-\nubtill-.5);

    \pgfmathsetmacro{\nuadjbtill}{floor((\cprime-\t-1)/2)}
    \pgfmathsetmacro{\rhsxshift}{floor(\c/2)+3}
    \draw (\rhsxshift-1,.5) node {$(k,l)$};
    \draw (\rhsxshift+\rhoplace,.5) node  {$\rho$};
    \draw[thin] (\rhsxshift+\seplinex,.7) -- (\rhsxshift+\seplinex,-\nuadjbtill-0.25);
    \draw[thin] (\rhsxshift-1.5,0.3) -- (\rhsxshift+\rhotill+.5,0.3);

    \pgfmathsetmacro{\muadja}{\mua}
    \pgfmathsetmacro{\muadjb}{\mub+1}

    \pgfmathsetmacro{\nubtill}{(\cprime-\t-1)/2}
    \foreach \nub in {0,...,\nubtill}
    {
    \pgfmathsetmacro{\nua}{\cprime-\t-1-\nub}
    \pgfmathsetmacro{\nuadescr}{int(\nua)}
    \draw (\rhsxshift-1,-\nub) node {$(\nuadescr,\nub)$};
    \foreach \rhob in {0,...,\rhotill}
    {
      \pgfmathsetmacro{\rhoa}{int(\c-\rhob)}
      \pgfmathsetmacro{\rhobdescr}{int(\rhob)}
      \draw (\rhob+\rhsxshift,-\nub) node [draw,rounded corners,inner sep=2pt] {$\rhoa,\rhobdescr$};
    }
    \pgfmathsetmacro{\rhobstart}{\muadjb+\nub}
    \pgfmathparse{min(\muadja+\nub,\muadjb+\nua)}
    \let\rhobend\pgfmathresult
    \foreach \rhob in {\rhobstart,...,\rhobend}
    {
    \pgfmathsetmacro{\rhoa}{int(\c-\rhob)}
    \pgfmathsetmacro{\rhobdescr}{int(\rhob)}
    \draw (\rhob+\rhsxshift,-\nub) node [draw,rounded corners,fill=black!20,inner sep=2pt] {$\rhoa,\rhobdescr$};
  }
  }

  \end{scope}
  \end{tikzpicture}
  \caption{Comparison of intervals $I\left( \mu_1,(k+1,l) \right)$ and
    $I\left( \mu_2,(k,l) \right)$---denoted by grey background---for
    adjacent partitions $\mu_1=(4,2) \vdash m'+2$ and $\mu_2=(4,3) \vdash m'+3$ of the 5-diagonal line
    segment in the $m=5$, $c'=8$ case}
  \label{tikz_diagonalsegmentintervals}
\end{figure}

\smallskip

\item\label{item_mon4horizontal}
Let $\mu_1=(i,j) \vdash m'+t$ and $\mu_2=(i+1,j) \vdash m'+(t+1)$ be adjacent partitions of a
vertical segment.
By inspecting intervals $I(\mu_1,\nu)$, $\nu \vdash h:=c'-t$ and $I\left( \mu_2,\nu \right)$,
$\nu \vdash c'-(t+1)=h-1$,
we will show that
\begin{equation}\label{eq_pi2frhocomparison_vertical}
  \pi_2\left( f_\rho(\mu_1) \right) \le \pi_2\left( f_\rho(\mu_2) \right)+1
\end{equation}
holds for every $\rho \vdash c$.

As the starting points of the corresponding intervals are equal,
\begin{equation*}\label{eq_horizontal_startingpt}
p\left( \mu_1,(k+1,l) \right)=(i+k+1,j+l)=p\left( \mu_2,(k,l) \right),
\end{equation*}
we can focus on their endpoints, or equivalently, their $\pi_2$-projections which we
will denote by
\[ g_1(l):=\pi_2\left( P\left( \mu_1, (h-l,l) \right) \right)=\min(i+l,j+h-l), \quad 0 \le l \le
\left\lfloor\frac{h}{2}\right\rfloor \]
and
\[ g_2(l):=\pi_2\left( P\left( \mu_2, (h-1-l,l) \right) \right)=\min(i+1+l,j+h-1-l), \quad 0 \le l \le
\left\lfloor\frac{h-1}{2}\right\rfloor. \]
Then $g_1(l+1)=g_2(l)$ for every $0 \le l <  \floor{h/2}$.

The above coincidence of starting points and endpoints,
together with \eqref{eq_unionalsointerval},
tells us that for every $0 \le l <  \floor{h/2}$
\[ \bigcup_{\nu_1 \le (k,l+1)} I\left( \mu_1,\nu_1 \right) \supset
\bigcup_{\nu_2 \le (k,l)} I\left( \mu_2,\nu_2 \right). \]
As a consequence,
\begin{equation}\label{eq_horizontal_unioninterval}
\pi_2(f_{\rho}(\mu_1)) \le \pi_2(f_{\rho}(\mu_2))+1 \text{ for every }
\rho \in \bigcup_{l < \floor{h/2}} I\left( \mu_2,(k,l) \right).
\end{equation}

What is left to prove (\ref{eq_pi2frhocomparison_vertical}) is that the union in \eqref{eq_horizontal_unioninterval}
contains all the $\rho$'s with $f_{\rho}(\mu_2) \neq \infty$.
In other words, that
 \begin{equation}\label{eq_horizontal_lastinterval}
   \bigcup_{ l \le \floor{(h-1)/2}} I\left( \mu_2, (k,l) \right)=
  \bigcup_{l < \floor{h/2} } I\left( \mu_2, (k,l) \right)
\end{equation}
(even when $\lfloor (h-1)/2 \rfloor=\lfloor h/2 \rfloor$, and there is an extra interval on the left-hand
side).

To accomplish this, we introduce 
\[ x_1:=x(\mu_1):=\frac{h+j-i}{2} \text{ and }  x_2:=x(\mu_2):=\frac{h-1+j-(i-1)}{2}=x_1-1, \]
elements where the $i+l$, $j+h-l$ arguments of $g_1(l)$ and the $i+1+l$, $j+h-1-l$ arguments of $g_2$
intersect respectively.
At the points $\floor{x_i}$ and $\ceil{x_i}$ (if nonnegative) $g_i$ takes its highest possible value, $\floor{h/2}$.
For $l \ge \floor{x_i}$ $g_i$ is monotone decreasing,
see Figure \ref{tikz_endpoints}.
Therefore
\begin{equation}\label{eq_xestimate}
  \floor{x_2}<\floor{x_1} \le \left\lfloor{\frac{h}{2}}\right\rfloor
\end{equation}
shows that
even if $\floor{h/2}=\floor{(h-1)/2}$ and there is an extra interval, 
its endpoint is smaller:
\[ g_2\left(\left\lfloor\frac{h-1}{2}\right\rfloor\right) \le
g_2\left(\left\lfloor\frac{h-1}{2}\right\rfloor-1\right),\]
therefore (\ref{eq_horizontal_lastinterval}) holds.
\pgfmathsetmacro{\m}{5}
\pgfmathsetmacro{\cprime}{8}
\pgfmathsetmacro{\c}{\m-1+\cprime}
\pgfmathsetmacro{\rhobtill}{\c/2}

\pgfmathsetmacro{\t}{1}
\pgfmathsetmacro{\h}{\cprime-\t}
\pgfmathsetmacro{\nubtill}{floor(\h/2)}
\pgfmathsetmacro{\nuadjbtill}{floor((\h-1)/2)}

\pgfmathsetmacro{\mua}{3}
\pgfmathsetmacro{\mub}{\m-1+\t-\mua}
\pgfmathsetmacro{\muadja}{\mua+1}
\pgfmathsetmacro{\muadjb}{\mub}

\pgfmathsetmacro{\gstartx}{min(\mua,\mub+\nubtill)}
\pgfmathsetmacro{\gmidy}{(\h+\mub-\mua)/2}
\pgfmathsetmacro{\gmidx}{\mua+\gmidy} 
\pgfmathsetmacro{\gendx}{min(\mua+\nubtill,\mub+\h-\nubtill)}

\pgfmathsetmacro{\gadjstartx}{min(\muadja,\muadjb+\nuadjbtill)}
\pgfmathsetmacro{\gadjmidy}{(\h-1+\muadjb-\muadja)/2}
\pgfmathsetmacro{\gadjmidx}{\muadja+\gadjmidy} 
\pgfmathsetmacro{\gadjendx}{min(\muadja+\nuadjbtill,\muadjb+\h-1-\nuadjbtill)}

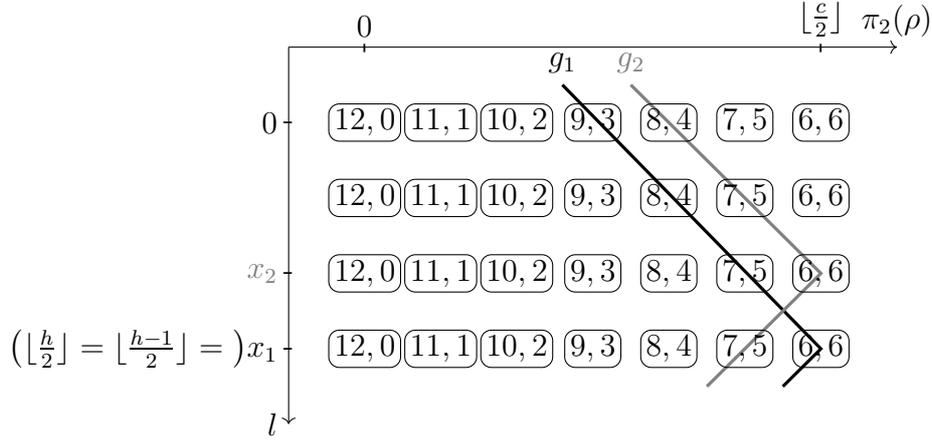
\begin{figure}
  \centering
  \begin{tikzpicture}
    \draw [very thick] (\gstartx-.4,0+.5) node [above,sloped] {$g_1$} -- (\gmidx,-\gmidy) -- (\gendx-.5,-\nubtill-.5);
    \draw [very thick,black!50] (\gadjstartx-.5,0+.5) node [above,sloped] {$g_2$} -- (\gadjmidx,-\gadjmidy) -- (\gadjendx-.5,-\nuadjbtill-.5);

    \draw[->] (-1,1) -- (-1,0) node [left] {$0$} node  {-}
    -- (-1,-\gadjmidy) node [left,black!50] {$x_2$} node  {-}
    -- (-1,-\gmidy) node [left] {$\big( \lfloor \frac{h}{2} \rfloor = \lfloor \frac{h-1}{2} \rfloor  =\big) x_1$} node  {-}
    -- (-1,-\nubtill-1) node [left] {$l$}; 

    \draw[->] (-1,1) --
    (0,1) node [rotate=90] {-} node [above] {$0$} --
    (\rhobtill,1) node [rotate=90] {-} node [above] {$\lfloor \frac{c}{2} \rfloor$} --
    (\rhobtill+1,1) node [above] {$\pi_2(\rho)$};

    \foreach \nub in {0,...,\nubtill}
    {
      \foreach \rhob in {0,...,\rhobtill}
      \pgfmathsetmacro{\rhoa}{int(\c-\rhob)}
      \pgfmathsetmacro{\rhobdescr}{int(\rhob)}
      \draw (\rhob,-\nub) node [draw,rounded corners,inner sep=2pt] {$\rhoa,\rhobdescr$};
    }
  \end{tikzpicture}
  \caption{Comparison of endpoints of intervals, $g_1(l)=P\left( \mu_1,(k+1,l) \right)$ and $g_2(l)=P\left( \mu_2,(k,l)
    \right)$ for  adjacent partitions $\mu_1=(3,2) \vdash m'+1$ and $\mu_2=(4,2) \vdash m'+2$ of the 2-vertical
  line segment in the $m=5$, $c'=8$ case}
  \label{tikz_endpoints}
\end{figure}

\end{enumerate}

\medskip

\item\label{item_puttingtogether}
We will finish proving (\ref{eq_thrm1intermsoft}) by showing---using induction on $t \ge 1$---that for every $\rho \vdash c$
and $\mu \vdash m'+t$
\begin{equation}\label{eq_degcomparedtot0n}
  \deg(r_{\mu,\rho}) < \deg\left( \text{coefficients of } s_\rho \text{ in }  A_0 B_0 \right).
\end{equation}

The induction step is outlined as follows.
For any $\mu \vdash m'+t$ ($t \ge 1$) with $p_{\mu} \neq 0$ there exists a diagonal (or a vertical) line segment
containing $\mu$, along which---except possibly for $\mu=\left(m,m \right)$, see later---$\mu$ is adjacent
to some $\mu_0 \vdash m'+(t-1)$ (Figure \ref{tikz_AtBtoverview}).
Either $\mu_0 \vdash m'$ or, by the induction hypothesis, \eqref{eq_degcomparedtot0n} holds for $\mu_0$.
For such adjacent partitions we proved, see \eqref{eq_pi2frhocomparison_diagonal} (or
\eqref{eq_pi2frhocomparison_vertical}) that $\pi_2(f_\rho(\mu_0)) \leq \pi_2(f_\rho(\mu))$
(or $\pi_2(f_\rho(\mu_0)) \leq \pi_2(f_\rho(\mu))+1$).
Therefore, we can use \eqref{eq_degsrhocoeff} to show that
\begin{equation*}
  \deg(r_{\mu,\rho}) <  \deg(r_{\mu_0,\rho}) \quad
  \Big( \text{or } \deg(r_{\mu,\rho}) \le   \deg(r_{\mu_0,\rho}) \Big).
\end{equation*}	

To clear up the ambiguity about the two types of line segments, let us first note that
if $\mu \vdash m'+t$ is contained in both an $f$-diagonal and a $g$-vertical line segment, that is
$p_\mu=e_f-e_g$, then $f > g$, hence $\deg(p_\mu)=f$.
This means that
if, in this case, we want to estimate $\deg(r_{\mu,\rho})$
via \eqref{eq_degsrhocoeff}, we have to use the diagonal line segment.

If the line segment can be chosen diagonal---as is the case for all $\mu \vdash m'+1$---,
then \eqref{eq_pi2frhocomparison_diagonal} combined with 
\eqref{eq_degsrhocoeff} becomes
\[
  \deg\left( r_{\mu,\rho}\right) <
  \deg\left( r_{\mu_0,\rho}\right).
\]

Moreover, if $t=1$, any possible $p_{\mu_0}$ ($\mu_0 \vdash m'$) has positive leading coefficient,
which, together with the positivity of the leading coefficients of the $q_\nu$'s ($\nu \vdash c'$) means that
high $d$-degree terms of the $r_{{\mu_0},\rho}$'s cannot cancel each other out.
This shows that for every $\mu_0 \vdash m'$
\[ \deg\left( r_{\mu_0,\rho}\right) \le 
  \deg\left( \text{coefficient of } s_\rho \text{ in }
 A_{0} B_{0}\right), \]
hence that \eqref{eq_degcomparedtot0n} holds for every $\mu \vdash m'+t$ in the $t=1$ base case.

If the line segment for $\mu \vdash m'+t$ can only be chosen to be a $g$-vertical line segment,
then $\deg(p_\mu)=g$ and $\deg(p_{\mu_0})\ge g$, hence \eqref{eq_pi2frhocomparison_vertical} combined with 
\eqref{eq_degsrhocoeff} becomes
\[
  \deg\left( r_{\mu,\rho}\right) \le   \deg\left( r_{\mu_0,\rho}\right).
\]
By the induction hypothesis for $\mu_0 \vdash m'+(t-1)$, this proves \eqref{eq_degcomparedtot0n}.  

Only the partition $\mu=(m,m) \vdash m'+m+1$ can have $p_\mu \neq 0$
while having no neighbour $\mu_0 \vdash m'+(t-1)$
along any line segment. It appears only if $c' \geq m+1$.
We compare $\deg(r_{\mu,\rho})$ to $\deg(r_{\mu_0,\rho})$ for $\mu_0=(m,m-1) \vdash m'+m$
the same way as if there was a diagonal line segment connecting them:
Analogously to \ref{item_mon4diagonal}, $\pi_2(f_\rho(\mu_0)) \le \pi_2(f_\rho(\mu))$,
$\deg(p_\mu)=\deg(p_{\mu_0})+1$, therefore by \eqref{eq_degsrhocoeff}
\[ \deg(r_{\mu,\rho}) \le \deg(r_{\mu_0,\rho}). \]
As $m+1 \ge 3$, we can apply the induction hypothesis to $\mu_0$, and get \eqref{eq_degcomparedtot0n} for $\mu$ as a result.

\medskip

\item\label{item_thrm1ii}
We will prove Theorem \ref{thrm_leadingtermfromproduct}/\ref{item_thrmii} by
showing that if $m-2 \le c'$, then for every $\rho \vdash c$
$\deg(r_{\mu,\rho})$ is strictly monotone decreasing in $\mu \vdash m'$, therefore
it attains its greatest value for the smallest partition, $\mu_{\min}=(m',0)$.

First we show that for every $\rho \vdash c$ the coefficient $r_{\mu_{\min},\rho} \neq 0$ , in other words, that
\begin{equation}\label{eq_mprime0intervals}
  \bigcup_{\nu \vdash c'} I\left( \mu_{\min},\nu \right)=\left\{\left. \rho \right| \rho \vdash c\right\}.
\end{equation}
The hypothesis  $m-2 \le c'$ is equivalent to
\[x\left( \mu_{\min} \right)=\frac{c'-m'}{2} > -1, \]
see \ref{item_mon4horizontal} for the definition and properties of $x(\mu)$.
This ensures that the set of endpoints
$\left\{ \left. P\left( \mu_{\min},\nu \right) \right| \nu \vdash c' \right\}$ contains
  the maximum, $(\ceil{c/2},\floor{c/2})$. As $p\left( \mu_{\min},(c',0) \right)=(c,0)$, (\ref{eq_mprime0intervals})
holds.

The monotonicity will follow from an analysis similar to that in \ref{item_mon4horizontal}.
More precisely, we will prove that for any $\rho \vdash c$ and adjacent partitions
$\mu_1=(i,j), \mu_2=(i-1,j+1) \vdash c'$
\begin{equation}\label{eq_verticaltzero}
  \pi_2\left( f_\rho(\mu_1) \right)  \le \pi_2\left( f_\rho(\mu_2) \right).
\end{equation}
As, by \eqref{eq_YmSchurcoeffleadingterm}, $\deg(p_{\mu_1})=m-j$ and $\deg(p_{\mu_2})=m-j-1$, the inequality (\ref{eq_degsrhocoeff})
then becomes
\begin{equation*}
  \deg\left( r_{\mu_1,\rho} \right) \ge \deg\left( r_{\mu_2,\rho} \right) +1,
\end{equation*}
showing the strictly monotone decreasing property.

In the comparison of intervals $I\left( \mu_1,(c'-l,l) \right)$ and $I\left( \mu_2,(c'-l,l) \right)$, for their
starting points we have
\[ p\left( \mu_1,(c'-l,l) \right) < p\left( \mu_2,(c'-l,l) \right). \]
To investigate their endpoints, we again use
\[ g_1(l):=\pi_2\left( P\left( \mu_1, (c'-l,l) \right) \right) \enskip \text{ and } \enskip
  g_2(l):=\pi_2\left( P\left( \mu_2, (c'-l,l) \right) \right),
\quad 0 \le l \le \left\lfloor\frac{c'}{2}\right\rfloor, \]
for which $g_1(l)=g_2(l+1)$, $ 0 \le l < \floor{c'/2} $, and
\[ x_1:=x(\mu_1)=\frac{c'+j-i}{2} \text{ and }  x_2:=x(\mu_2)=\frac{c'+j+1-(i-1)}{2}, \]
for which $x_1=x_2-1<\floor{c'/2}$.
These imply that
for every $\nu \vdash c'$
\[ \bigcup_{\nu_1 \le \nu} I\left( \mu_1,\nu_1 \right) \supsetneqq
\bigcup_{\nu_2 \le \nu} I\left( \mu_2,\nu_2 \right), \]
therefore (\ref{eq_verticaltzero}) holds.

\medskip

\item\label{item_thrm2}
To prove Theorem \ref{thrm_rrholeadterm}, let us choose $m=\min{\lambda}$.
That is $\lambda=(m^{e_m},\dots,r^{e_r})$.
Then $m-1=m' \le c'$, so we
can use Theorem \ref{thrm_leadingtermfromproduct}/\ref{item_thrmii}, and get that
for every $\rho \vdash c$
the leading term of $r_\rho$ comes from $r_{\mu_{\min},\rho}$, where, again,
we use the notation $\mu_{\min}=(m',0)$.
In particular, this, combined with \eqref{eq_degformu}, gives that
\begin{equation}\label{eq_degrrho}
  \deg(r_\rho)= \deg(p_{\mu_{\min}})+\deg(q_{f_\rho(\mu_{\min})})=m+\deg(q_{f_\rho(\mu_{\min})}).
\end{equation}

To describe $\deg(q_{f_\rho(\mu_{\min})})$, we first look into the function $\rho=(c-v,v) \mapsto
f_{(c-v,v)}(\mu_{\min})$, or equivalently, it's $\pi_2$-projection. In this $m-1 \le c'$ case this can
be easily computed from the intervals $\left\{I(\mu_{\min},\nu)\right\}_{\nu \vdash c'}$ to be
\begin{equation*}\label{eq_frhofunction}
  \pi_2(f_{(c-v,v)}(\mu_{\min}))=\max(0,v-m').
\end{equation*}

The coefficients $q_{f_\rho(\mu_{\min})}$ are coefficients of
$\left[ \, \overline Y_{\lambda'}(d) \right]$, hence, by the induction hypothesis for $\lambda'$,
there is a threshold $\vartheta(\lambda')$ for their behaviour.
By the monotonocity of $\rho \mapsto f_{\rho}(\mu_{\min})$, this means that for
\begin{equation}\label{eq_thetalambda_def}
 \vartheta(\lambda):=
  \max\left\{ \left. \rho \vdash c \, \right| f_{\rho}(\mu_{\min}) \le \vartheta(\lambda') \right\} 
  \end{equation}
we have $\rho \le \vartheta(\lambda)$ if and only if $f_\rho(\mu_{\min}) \le \vartheta(\lambda')$, see
Figure \ref{tikz_degqfrho_composition}.
\begin{figure}[H]
  \centering
  \pgfmathsetmacro{\addlambda}{16}
\pgfmathsetmacro{\c}{13}
\pgfmathsetmacro{\lambdamax}{10}
\pgfmathsetmacro{\lambdamin}{3}
\pgfmathsetmacro{\c}{14}
\pgfmathsetmacro{\cprime}{\c-(\lambdamin-1)}
\pgfmathsetmacro{\rhobtill}{floor(\c/2)}
\pgfmathsetmacro{\nubtill}{floor(\cprime/2)}
\pgfmathsetmacro{\thetalambdaprime}{\cprime-(\lambdamax)+1}

\pgfmathsetmacro{\addlambdaprime}{13}
\pgfmathsetmacro{\lambdaprimemax}{10}

\pgfmathsetmacro{\xcutto}{1}
\pgfmathsetmacro{\xcutdiff}{8}
\begin{tikzpicture}

\draw[->]
(0,0)  --
(0,\thetalambdaprime) node [left] {$\vartheta(\lambda')$} --
(0,\nubtill+1) node [left] {$\nu$} ;

\draw[->]
(0,0) node [below] {$0$} --
(\rhobtill,0) node [rotate=90] {-} node [below] {$\lfloor \frac{c}{2} \rfloor$} --
(\lambdamin-1,0) node [rotate=90] {-} node [below] {$m-1$} --
(\lambdamin-1+\thetalambdaprime,0) node [below] {$\vartheta(\lambda)$} --
(\rhobtill+1,0) node [below] {$\pi_2(\rho)$};
\foreach \rhob in {0,...,\rhobtill}
{
\draw (\rhob,0) node [rotate=90] {-} ;
}

\foreach \rhob in {0,...,\rhobtill}
{
  \pgfmathsetmacro{\frhomub}{max(0,-(\lambdamin-1)+\rhob)}
\draw (\rhob,\frhomub) node [circle,fill,inner sep=2pt] {};
}

\draw (\rhobtill,\rhobtill-\lambdamin+1+.5) node [above,rotate=45] {$f_{\rho}(\mu_{\min})$};

\draw[dashed] (0,\thetalambdaprime) -- (\lambdamin-1+\thetalambdaprime,\thetalambdaprime);
\draw[dashed] (\lambdamin-1+\thetalambdaprime,0) -- (\lambdamin-1+\thetalambdaprime,\thetalambdaprime);

\begin{scope}[shift={(\rhobtill+2,0)}]
\draw[->]
(0,0)  --
(0,\thetalambdaprime) node [left] {$\vartheta(\lambda')$} --
(0,\nubtill+1) node [left] {$\nu$} ;
\foreach \nub in {0,...,\nubtill}
{
\draw (0,\nub) node{-} ;
}

\draw[loosely dotted,thick] (0,0) -- (\xcutto,0);
\draw[->]
(\xcutto,0)  --
 (\addlambdaprime-\xcutdiff,0) node [below] {$|\lambda'|$} node [rotate=90] {-} --
(\addlambdaprime-\xcutdiff+1,0) ;
\pgfmathsetmacro{\dashstart}{\xcutto+1}
\pgfmathsetmacro{\dashtill}{\addlambdaprime-\xcutdiff}
\foreach \nub in {\dashstart,...,\dashtill}
{
  \draw (\nub,0) node [rotate=90] {-};
}

\foreach \nub in {0,...,\nubtill}
{
\pgfmathsetmacro{\degqnu}{min(\addlambdaprime,\addlambdaprime+\thetalambdaprime-\nub)}
\draw (\degqnu-\xcutdiff,\nub) node [circle,fill,inner sep=2pt] {};
}

\draw[dashed] (0,\thetalambdaprime) -- (\addlambdaprime-\xcutdiff,\thetalambdaprime);

\pgfmathsetmacro{\degqmulabel}{min(\addlambdaprime,\addlambdaprime+\thetalambdaprime-\nubtill)}
\draw (\degqmulabel-\xcutdiff+.5,\nubtill) node [above,rotate=-45] {$\deg(q_\nu)$};
\end{scope}

\pgfmathsetmacro{\rhobspecific}{4}
\pgfmathsetmacro{\frhomubspecific}{max(0,-(\lambdamin-1)+\rhobspecific)}
\pgfmathsetmacro{\nubspecific}{\frhomubspecific}
\pgfmathsetmacro{\degqnuspecific}{min(\addlambdaprime,\addlambdaprime+\thetalambdaprime-\nubspecific)}
\draw[magenta] (\rhobspecific,0) -- node [rotate=90] {\midarrow} (\rhobspecific,\frhomubspecific);
\draw[magenta,dashed] (\rhobspecific,\frhomubspecific) -- node (\rhobtill,\frhomubspecific) [above] {$\deg(f_{\rho}(\mu_{\min}))$}
(\rhobtill+2,\frhomubspecific);
\draw[magenta] (\rhobtill+2,\frhomubspecific) -- node {\midarrow} (\degqnuspecific-\xcutdiff+\rhobtill+2,\frhomubspecific);

\end{tikzpicture}
  \caption{Functions $\rho \mapsto f_\rho(\mu_{\min})$ and
    $\nu \mapsto \deg(q_\nu)$ in the composition $\deg(q_{f_\rho(\mu_{\min})})$ together with 
    the thresholds $\vartheta(\lambda')$ and $\vartheta(\lambda)$ for 
  $\lambda=(10,3,3)$ (and $\lambda'=(10,3)$, $\mu_{\min}=(2,0)$, etc.) }
  \label{tikz_degqfrho_composition}
\end{figure}
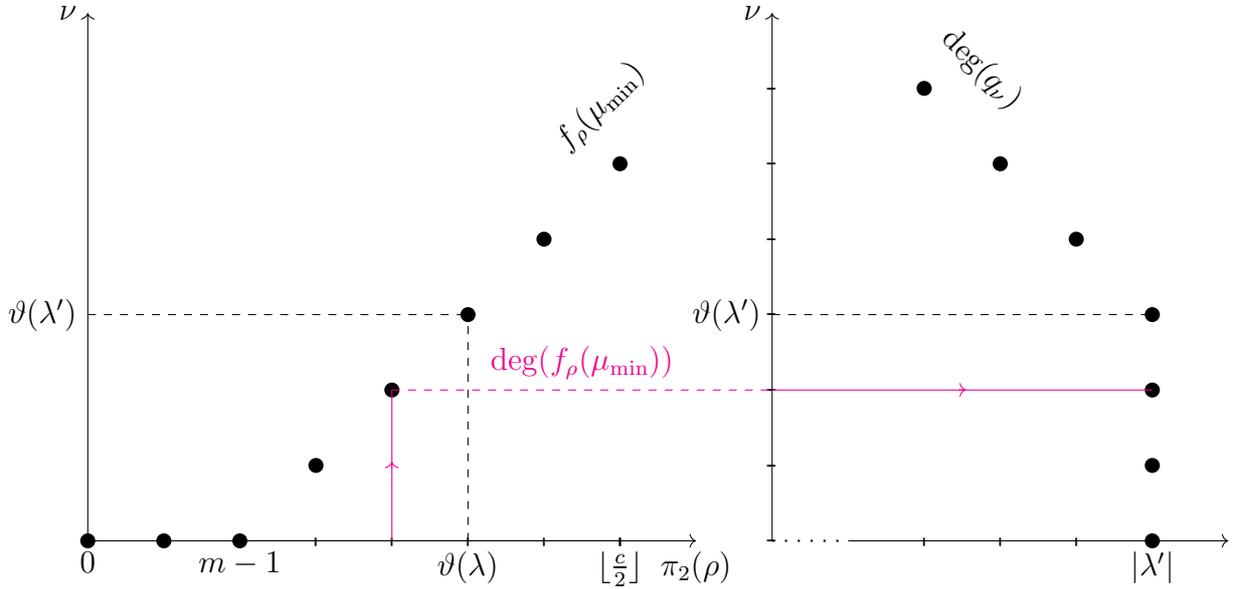
Then for every $\rho \le \vartheta(\lambda)$ \eqref{eq_degrrho} combined with the induction hypothesis for
$\lambda'$ gives that
$\deg(r_\rho)=m+|\lambda'|=|\lambda|$ and hence, by Theorem \ref{thrm_Kostkaleadingcoeffs}, the leading term of $r_\rho$ is
\[ \frac{K_{\rho,\tilde{\lambda}}}{\prod_{i=m}^r \left(e_i!\right)} \, d^{|\lambda|}. \]

If $\rho > \vartheta(\lambda)$, then $\pi_2(\rho) - \pi_2(\vartheta(\lambda))=
\pi_2(f_{\rho}(\mu_{\min}))-\pi_2(\vartheta(\lambda'))$, therefore \eqref{eq_degrrho} and the induction hypothesis imply that
\begin{multline*}
  \deg(r_\rho)=m+ \deg(q_{f_\rho(\mu_{\min})})=m+ \left( |\lambda'|- \left( \pi_2(f_{\rho}(\mu_{\min}))-\pi_2(\vartheta(\lambda'))
  \right) \right)=\\
|\lambda|-\left( \pi_2(\rho) - \pi_2(\vartheta(\lambda)) \right). 
\end{multline*}
To obtain the leading coefficient of $r_\rho$ for $\rho > \vartheta(\lambda)$, let us note that,
as $\nu \mapsto \deg(q_\nu)$ is strictly monotone decreasing for $\nu > \vartheta(\lambda')$,
the set
\[ \left\{ \left. \nu \vdash c' \right| \rho \in I(\mu_{\min},\nu)  \right\} \]
has a unique element $\nu\left( = f_{\rho}(\mu_{\min}) \right)$ with maximal $\deg(q_\nu)\left( =\deg(q_{f_\rho(\mu_{\min})})
\right)$. This means that the leading term of $r_\rho$ must come from the summand
\[ \frac{1}{e_m} p_{\mu_{\min}} s_{\mu_{\min}} q_{f_{\rho}(\mu_{\min})} s_{f_{\rho}(\mu_{\min})} \]
summand.
In particular, the leading coefficient of $r_\rho$ is $\left(1/e_m\right)$-times the product of those of $p_{\mu_{\min}}$
and $q_{f_{\rho}(\mu_{\min})}$.
Since we chose $m=\min(\lambda)$, the largest elements of $\lambda$ and $\lambda'$ are both $\lambda_1$.
Therefore, by the induction hypotheses, these leading coefficients are $1$ and
\begin{multline*}
\frac{1}{ \left( e_m-1 \right)! \prod_{i=m+1}^{r} \left( e_i! \right)}
  \stir{\lambda_1}{\lambda_1-\left( \pi_2(f_{\rho}(\mu_{\min}))-\pi_2(\vartheta(\lambda')) \right)}=\\
\frac{1}{ \left( e_m-1 \right)! \prod_{i=m+1}^{r} \left( e_i! \right)}
\stir{\lambda_1}{\lambda_1-\left( \pi_2(\rho)-\pi_2(\vartheta(\lambda)) \right)}. 
\end{multline*}

To complete the induction step, all we need to check is that
the ($\pi_2$-projection of) $\vartheta(\lambda)$ defined in \eqref{eq_thetalambda_def} agrees with the
one in Theorem
\ref{thrm_rrholeadterm}.
Because of the induction hypothesis,
\[ \pi_2\left(\vartheta(\lambda')\right)=
\min\left( \left\lfloor \frac{c'}{2} \right\rfloor, c'-\lambda_1'+1 \right), \]
and the description of the function $\rho \mapsto f_{\rho}(\mu_{\min})$, see also Figure \ref{tikz_degqfrho_composition}, all we need to check is
that
\[
  \pi_2(\vartheta(\lambda))=
\min\left( \left\lfloor \frac{c}{2} \right\rfloor, m'+
\min\left( \left\lfloor \frac{c'}{2} \right\rfloor, c'-\lambda_1+1 \right) \right)=
\min\left( \left\lfloor \frac{c}{2}\right\rfloor, c-\lambda_1+1  \right)\!.
\]
This follows easily from the observation that adding $m'$ to the inequality
\[ \left\lfloor \frac{c'}{2} \right\rfloor \le c'-\lambda_1+1, \]
we get
\[ \left\lfloor \frac{c}{2} \right\rfloor = \left\lfloor \frac{c'+m'}{2} \right\rfloor \le
\left\lfloor \frac{c'}{2} \right\rfloor+m' \le c'-\lambda_1+1+m'=c-\lambda_1+1. \]

\end{enumerate}

\end{proof}

\appendix
\section{Polynomiality of \texorpdfstring{$r_{\mu,\rho}(d)$}{r mu,rho(d)}} \label{sec_polynomiality_rmurho}
In this Appendix we show
\begin{proposition}\label{prop_rmurho_ispoly}
  The a priori rational function coefficients $r_{\mu,\rho}(d)$ defined in \eqref{eq_pmuterm} (and expanded
    in \eqref{eq_rmurho_equals}) are polynomials in $d$.
\end{proposition}

This is not strictly necessary for the proof in Section
\ref{sec_proofdegreesofPlnumbers} to work: Setting the degree of a rational function $f=p/q$
$\left( p, q \in \mathbb{Q}[d]\right)$ to $\deg(f)=\deg(p)-\deg(q)$,
the proof works without any modification.

Including this proof, however, we can highlight the fact that the coefficients
$q_{\nu}(d)/d^t$ ($\nu \vdash c'-t$) have a geometric interpretation, see \eqref{eq_qmuperdt_enumerative}.
Similarly to how we proved Theorem \ref{Y-poly} (or, equivalently, Theorem
\ref{thrm_pluckernumberspolynomial}) in \cite{feher_juhasz2023plucker}, this leads to
\begin{lemma}\label{lemma_dtdividesqmu}
  For every $\nu \vdash c'-t$ ($t=0,\dots,c'$) for the coefficient $q_\nu \in \mathbb{Q}[d]$
  of the Schur polynomial $s_\nu$ in $B_t$, see \eqref{eq_pmuqnudef}, we have
  \[ d^t \mid q_{\nu}(d). \]
\end{lemma}
By \eqref{eq_rmurho_equals}, this proves Proposition \ref{prop_rmurho_ispoly}.

\begin{proof}[Proof of Lemma \ref{lemma_dtdividesqmu}]
In \cite[\S~8.1]{feher_juhasz2023plucker} we showed that 
the class of the set 
\[ \overline{\mathcal{T}_\lambda(d)}:=
  \left\{ \left. \left( \left[ f \right],W \right) \in \P(\Pol^d(\C^n)) \times \Gr_2(\C^n)
    \right| W \in \overline{\mathcal{T}_\lambda Z_f}  \right\}
  \]
  can be deduced from $\left[ \, \overline Y_{\lambda}(d) \right]$ as
  \begin{equation}\label{eq_classoflocus_universalsection}
  \left[ \, \overline{\mathcal{T}_\lambda(d)} \subset \P(\Pol^d(\C^n)) \times \Gr_2(\C^n) \right]=
\left[ \, \overline Y_{\lambda}(d) \right]
\big\rvert_{\stackon{$\hspace{.075em} \scriptstyle b \mapsto b + (1/d)x$}{$\scriptstyle a  \mapsto  a+(1/d)x$}},
\end{equation}
where on the right-hand side of the substitution
$a,b$ and $x$ denote the Chern roots of the duals of the tautological bundles $S \to \Gr_2(\C^n))$
and $L \to \P(\Pol^d(\C^n))$.

Expanding the right-hand side of \eqref{eq_classoflocus_universalsection} for $\lambda=\lambda'$,
we get that 
for big enough $n$'s ($n \ge c'+2$)
\begin{equation}\label{eq_class_locusofuniversalsection}
\left[ \, \overline Y_{\lambda'}(d) \right]
\big\rvert_{\stackon{$\hspace{.075em} \scriptstyle b \mapsto b + (1/d)x$}
{$\scriptstyle a  \mapsto  a+(1/d)x$}}=
\sum_{t=0}^{c'}\left( \frac{1}{d} \right)^t B_t x^t = 
\sum_{t=0}^{c'}\sum_{\nu \vdash c'-t} \left( \frac{1}{d} \right)^t q_{\nu}(d) x^t s_{\nu},
\end{equation}
where all the $B_t$'s and the $q_{\nu}$'s are as in \eqref{eq_Bt_def} and \eqref{eq_pmuqnudef} but
now understood in the variables the Chern roots $a$ and $b$ of $S^{\vee} \to \Gr_2(\C^n)$.
Using Schubert calculus, we see that the coefficient of an $x^t s_{\nu}$ ($\nu \vdash c'-t$) solves
an enumerative problem:
If $\nu=(u,v)$, then
for a generic $t$-dimensional linear system $S$ of degree $d$ hypersurfaces in $\P(\C^n)$ and generic
linear subspaces
$A$ of dimension $v$ and $B$ of dimension $u+1$ such that $A \subset B \subset \P(\C^n)$ 
\begin{multline}\label{eq_qmuperdt_enumerative}
\left( \frac{1}{d} \right)^t q_{(u,v)} (d)=\\
  \text{the number of }
  \lambda'\text{-lines to a member of } S \text{ that intersect } A \text{ and are contained in } B.
\end{multline}
In particular, the values of the rational functions $q_{\nu}(d)/d^t$ ($\nu \vdash c'-t$) are all integers for
$d \gg 0$, hence the following well-known Lemma finishes the proof.
\begin{lemma}\label{rat-poly}
Suppose that $f(x)$ is a rational function, such that $f(d)$ is an integer for all $d \gg 0$ integers.
Then $f(x)$ is a polynomial.
\end{lemma}
\end{proof}

\bibliography{bibliography_andris}

\begin{thebibliography}{FNR06}

\bibitem[Col86]{colley1986contact}
S.~J. Colley.
\newblock Lines having specified contact with projective varieties.
\newblock In {\em Proceedings of the 1984 Vancouver Conference in Algebraic
  Geometry}, Annual seminar of the Canadian Mathematical Society, pages 47--70.
  American Mathematical Society, 1986.

\bibitem[FJ23]{feher_juhasz2023plucker}
L\'aszl\'o~M. Feh\'er and Andr\'as~P. Juh\'asz.
\newblock Pl\"ucker formulas using equivariant cohomology of coincident root
  strata, 2023.
\newblock \url{https://arxiv.org/pdf/2312.06430.pdf}.

\bibitem[FNR06]{fnr-root}
L.~M. Feh\'{e}r, A.~N\'{e}methi, and R.~Rim\'{a}nyi.
\newblock Coincident root loci of binary forms.
\newblock {\em Michigan Math. J.}, 54(2):375--392, 2006.

\bibitem[Kaz06]{kazarian}
M.~Kazarian.
\newblock Morin maps and their characteristic classes.
\newblock
  \url{https://citeseerx.ist.psu.edu/viewdoc/download?doi=10.1.1.486.8275&rep=rep1&type=pdf},
  2006.

\bibitem[K\H03]{balazs-tezis}
B.~K\H{o}m\H{u}ves.
\newblock Thom polynomials via restriction equations.
\newblock Master's thesis, E\"{o}tv\"{o}s L\'{o}r\'{a}nt University Budapest,
  2003.

\bibitem[Kir84]{kirwan}
F.~Kirwan.
\newblock {\em Cohomology of quotients in symplectic and algebraic geometry}.
\newblock Number~31 in Mathematical Notes. Princeton UP, 1984.

\bibitem[Kle77]{kleiman1977enumtheorysingularities}
S.~L. Kleiman.
\newblock The enumerative theory of singularities.
\newblock {\em Uspekhi Matematicheskikh Nauk}, 1977.

\bibitem[Kle81]{kleiman1981multiplepoint_iteration}
S.~L. Kleiman.
\newblock {Multiple-point formulas I: Iteration}.
\newblock {\em Acta Mathematica}, 147:13 -- 49, 1981.

\bibitem[Kle82]{Kleiman1982multiplepoint_formaps}
S.~L. Kleiman.
\newblock {\em Multiple Point Formulas for Maps}, pages 237--252.
\newblock Birkh{\"a}user Boston, Boston, MA, 1982.

\bibitem[LB82]{lebarz-formules}
Patrick Le~Barz.
\newblock Formules multis\'{e}cantes pour les courbes gauches quelconques.
\newblock In {\em Enumerative geometry and classical algebraic geometry
  ({N}ice, 1981)}, volume~24 of {\em Progr. Math.}, pages 165--197.
  Birkh\"{a}user, Boston, Mass., 1982.

\bibitem[ST22]{spink-tseng}
Hunter Spink and Dennis Tseng.
\newblock {$PGL_2$}-equivariant strata of point configurations in {$\mathbb
  P^1$}.
\newblock {\em Ann. Sc. Norm. Super. Pisa Cl. Sci. (5)}, 23(2):569--621, 2022.

\bibitem[Tot99]{totaro}
Burt Totaro.
\newblock The {C}how ring of a classifying space.
\newblock In {\em Algebraic {$K$}-theory ({S}eattle, {WA}, 1997)}, volume~67 of
  {\em Proc. Sympos. Pure Math.}, pages 249--281. Amer. Math. Soc., Providence,
  RI, 1999.

\end{thebibliography}
\bibliographystyle{alpha}
\end{document}